\documentclass[11pt]{article}
\usepackage{amsfonts}
\usepackage{mathrsfs}
\usepackage{bbm}
\usepackage{amsfonts}
\usepackage{amssymb,amsmath,amsthm,graphicx}
\usepackage{float}
\usepackage[colorlinks, linkcolor=red]{hyperref}
\usepackage{color, xcolor}
\usepackage{cite}
\usepackage[paper=a4paper, left=1.6cm, right=1.6cm, top=1.8cm, bottom=1.6cm, headheight=5.5pt, footskip=0.8cm, footnotesep=0.8cm, centering, includefoot]{geometry}

\hypersetup{urlcolor=red, citecolor=red}

\DeclareMathOperator{\divv}{div}
\DeclareMathOperator{\curl}{curl}

\allowdisplaybreaks

\begin{document}
\title{Global existence of strong solutions with large oscillations and vacuum to the compressible nematic liquid crystal flows in 3D bounded domains
\thanks{
This research was partially supported by National Natural Science Foundation of China (No. 12371227), Scientific Research Foundation of Jilin Provincial Education Department (No. JJKH20210873KJ), and Postdoctoral Science Foundation of China (No. 2021M691219).}
}
\author{Yang Liu$\,^{\rm 1},$\ Xin Zhong$\,^{\rm 2}\,${\thanks{Corresponding author. E-mail addresses: liuyang0405@ccsfu.edu.cn (Y. Liu), xzhong1014@amss.ac.cn (X. Zhong).}}
\date{}\\
\footnotesize $^{\rm 1}\,$
College of Mathematics, Changchun Normal
University, Changchun 130032, P. R. China\\
\footnotesize $^{\rm 2}\,$ School of Mathematics and Statistics, Southwest University, Chongqing 400715, P. R. China} \maketitle
\newtheorem{theorem}{Theorem}[section]
\newtheorem{definition}{Definition}[section]
\newtheorem{lemma}{Lemma}[section]
\newtheorem{proposition}{Proposition}[section]
\newtheorem{corollary}{Corollary}[section]
\newtheorem{remark}{Remark}[section]
\renewcommand{\theequation}{\thesection.\arabic{equation}}
\catcode`@=11 \@addtoreset{equation}{section} \catcode`@=12
\maketitle{}

\begin{abstract}
We investigate compressible nematic liquid crystal flows in three-dimensional (3D) bounded domains with slip boundary condition for velocity and Neumann boundary condition for orientation field. By applying piecewise-estimate method and delicate analysis based on the effective viscous flux and vorticity, we derive the global existence and uniqueness of strong solutions provided that the initial total energy is suitably small. Our result is an extension of the works of Huang--Wang--Wen (J. Differential Equations 252: 2222--2265, 2012) and Li--Xu--Zhang (J. Math. Fluid Mech. 20: 2105--2145, 2018), where the local strong solutions in three dimensions and the global strong solutions for 3D Cauchy problem were established, respectively. Moreover, it also shows that blow up mechanism for local strong solutions obtained by Huang--Wang--Wen (Arch. Ration. Mech. Anal. 204: 285--311, 2012) cannot occur if the initial total energy is sufficiently small.
\end{abstract}




\section{Introduction}

Liquid crystals can form and remain in an intermediate
phase of matter between liquids and solids. When a solid melts, if
the energy gain is enough to overcome the positional order but the
shape of the molecules prevents the immediate collapse of
orientational order, liquid crystals are formed. The nematic liquid crystals exhibit long-range ordering in the sense that their rigid rod-like molecules arrange themselves with their long axes parallel to each other. Their molecules float around as in a liquid, but have the tendency to align along a preferred direction due to their orientation. The continuum theory of the nematic liquid crystals was first developed by Ericksen \cite{E1962} and Leslie \cite{L1968} during the period of 1958 through 1968. The rigorous mathematical analysis of the Ericksen-Leslie model was first made by Lin \cite{L1989} and Lin-Liu \cite{LL1995,LL96,LL00}, in which they introduced a considerably simplified version and proved the existence of global weak solutions and their partial regularities. Regarding modeling and analysis of the Ericksen-Leslie equations describing nematic liquid crystals, please refer to
the survey papers \cite{HP18,LW14,Z21} as well as the references therein
for more discussions on the physics and mathematical results.

In this paper, let $\Omega\subset \mathbb{R}^3$ be a bounded domain, we consider a simplified hydrodynamic flow modeling compressible nematic liquid crystal materials in $\Omega\times(0,T)$:
\begin{align}
 \rho_t+\divv(\rho u)&=0,\label{a1}\\
\rho u_t+\rho u\cdot\nabla u+\nabla P
&=\mathcal Lu
-\Delta d\cdot\nabla d,\label{a2}\\
d_t+u\cdot \nabla d&=\Delta d+|\nabla d|^2d,\label{a3}
\end{align}
where $\rho:\Omega\times(0,T)\rightarrow \mathbb{R}^+$ is the density, $u:\Omega\times(0,T)\rightarrow\mathbb{R}^3$ is the velocity field, $d:\Omega\times(0,T)\rightarrow\mathbb{S}^2$ represents
the macroscopic average of the nematic liquid crystal orientation field, $P(\rho)=a\rho^\gamma\ (a>0, \gamma>1)$ is the pressure,
while $\mathcal L$ denotes the Lam\'e operator defined by
\begin{align*}
\mathcal Lu=(\lambda+\mu)\nabla \divv u+\mu\Delta u,
\end{align*}
where $\mu$ and $\lambda$ are the shear viscosity and the bulk viscosity coefficients of the fluid, respectively, which satisfy the physical restrictions
\begin{align*}
\mu>0, \quad 3\lambda+2\mu\ge 0.
\end{align*}
The system \eqref{a1}--\eqref{a3} will be studied along with
the initial condition
\begin{align}
(\rho, \rho u, d)(x, 0)=(\rho_0, \rho_0u_0, d_0)(x), \quad x\in\Omega,
\end{align}
and boundary condition for $(u,d)$:
\begin{align}\label{a6}
u\cdot n=0,~ \curl u\times n=0, \quad \frac{\partial d}{\partial n}=0, \quad {\rm on}\ \ \partial\Omega\times(0,T),
\end{align}
where $n=(n^1, n^2, n^3)$ is the unit outward normal vector to $\partial\Omega$.

From the mathematical point of view, system \eqref{a1}--\eqref{a3} is a strongly coupled system between the compressible Navier-Stokes equations (see, e.g., \cite{CJ21,NS04}) and the transported heat flows of harmonic map (see, e.g., \cite{W11}).
In the last few years, there have been substantial developments on the mathematical study for such a model. Among them, Wang and Yu \cite{WY12} obtained
existence and large-time behavior of a global weak solution through
a three-level approximation, energy estimates, and weak convergence for the adiabatic
exponent $\gamma>\frac32$.
Lin--Lai--Wang \cite{LLW15} established the existence of finite energy weak solutions with the large initial data provided that the initial orientational director field $d_0$ lies in $\mathbb{S}_2^+$.
When the initial data are of small energy and initial density is positive and essentially bounded, Wu and Tan \cite{WT18} proved the global existence of a weak solution in $\mathbb{R}^3$ and established the large-time behavior of such a weak solution.
At the same time, for 2D case, Jiang--Jiang--Wang \cite{JJW2013} investigated the existence of global weak solutions in a bounded domain under a restriction imposed on the initial energy. Moreover, they \cite{JJW14} also obtained the global existence of finite energy weak solutions to the 2D Cauchy problem under the assumption that the second component of initial data of direction field satisfied some geometric angle condition.

On the other hand, there are some results devoting to the existence of strong solutions.
Applying local well-posedness and uniform estimates for solutions of proper linearized systems with convective terms, Hu and Wu \cite{HW13} showed global existence and uniqueness of strong solutions in critical Besov spaces provided that the initial data are close to an equilibrium state $(1, 0, \bar d)$ with a constant vector $\bar d\in\mathbb{S}^2$. Gao--Tao--Yao \cite{GTY16} investigated the global existence of classical solution under the assumption that the initial data are close to the constant equilibrium state in $H^N (\mathbb{R}^3)\ (N\geq3)$-framework and proved algebraic time decay for the classical solution by weighted energy method. Schade and Shibata \cite{SS15} established the unique existence of local-in-time strong solutions in a uniform $W^{3-\frac1q}_q$-domain $\Omega \subseteq \mathbb R^N$ for $N<q<\infty$. Furthermore, if $\Omega$ is bounded and initial data are chosen suitably small, they derived global-in-time strong solutions.
When the initial density allows for vacuum (i.e. the initial density may
vanish in some open sets), the issue of the existence of solutions becomes much more complicated due to the possible degeneracy near vacuum. Under the condition that the initial data satisfies a compatibility condition
\begin{align}\label{1.8}
-\mu\Delta u_0-(\lambda+\mu)\nabla{\rm div}\,u_0+\nabla P_0+
\nabla d_0\cdot\Delta d_0=\sqrt{\rho_0}g
\end{align}
for some $g\in L^2$, Huang-Wang-Wen \cite{HWW2012} obtained the short existence and uniqueness of strong solutions to the 3D initial boundary value problem with vacuum. Later on, Li--Xu--Zhang \cite{LXZ18} established the unique global classical solutions to the 3D Cauchy problem with smooth initial data which are of small energy but possibly large oscillations with the constant state as the far-field condition. By weighted energy method, the global well-posedness of strong solutions containing vacuum in $\mathbb{R}^2$ was studied by Wang \cite{W16}. There are also other interesting studies on the compressible nematic liquid crystal flows, such as singularity formation of strong solutions \cite{HWW12,WH18}, incompressible limit \cite{DHWZ13,WY14}, and the optimal time-decay rates \cite{BWY18} and so on.

However, all of the results on global strong solutions mentioned above only concern with the whole space or with non-vacuum. It remains completely open for the global existence of strong solutions of compressible nematic liquid crystal flows with density containing vacuum initially in general bounded domains. So, motivated by the recent work due to Cai and Li \cite{CJ21}, where the authors proved global classical solutions with vacuum and small initial energy for the 3D initial-boundary-value problem of compressible isentropic Navier-Stokes equations with slip boundary condition, the main aim of this paper is to investigate the global well-posedness of strong solutions with large oscillations and vacuum to the initial-boundary-value problem \eqref{a1}--\eqref{a6} when the initial energy is suitably small.

Before stating our main result, we first explain the notation and conventions
used throughout the paper. We denote the initial total energy of \eqref{a1} by
\begin{align}
E_0\triangleq\int\Big(\frac12\rho_0|u_0|^2+\frac12|\nabla d_0|^2+G(\rho_0)\Big)dx,
\end{align}
where
\begin{align}\label{1.7}
G(\rho)\triangleq\rho\int_{\bar{\rho}}^\rho\frac{P(\xi)-\bar{P}}{\xi^2}d\xi, \quad \bar{\rho}\triangleq\frac{1}{|\Omega|}\int\rho_0dx,
\quad \bar{P}\triangleq P(\bar{\rho}).
\end{align}
Moreover, we write
\begin{align*}
&H^1_{\omega} \triangleq\left\{v\in H^{1}(\Omega): v\cdot n=0\ \text{and}\ \curl v\times n=0\ \text{on}\ \partial\Omega\right\},\\
&H_{n}^{2}\triangleq\{v\in H^{2}(\Omega): \nabla v\cdot n=0\ \text{on}\ \partial\Omega\}.
\end{align*}

Now we state our main result concerning global strong solutions to the problem \eqref{a1}--\eqref{a6}.
\begin{theorem}\label{thm1}
Let $\Omega$ be a bounded simply connected smooth domain in $\mathbb{R}^3$ and its boundary $\partial\Omega$ has a finite number
of two-dimensional connected components. For some given constants $M_1$, $M_2>0$ (not necessarily small) and $\hat{\rho}\ge \bar{\rho}+1$,
suppose that the initial data  $(\rho_0, u_0, d_0)$ satisfying, for $q\in (3, 6)$,
\begin{align}\label{a10}
\begin{cases}
0\le \rho_0\le \hat{\rho}, \ (\rho_0, P(\rho_0))\in W^{1, q},\ u_0\in H^1_{\omega},\ d_0\in H_n^2, \\
\|\nabla u_0\|_{L^2}^2\le M_1, \ \|\Delta d_0\|_{L^2}^2\le M_2,\ |d_0|=1.
\end{cases}
\end{align}
 There exists a positive constant $\varepsilon$ depending only on $\mu$, $\lambda$,
$\gamma$, $a$, $\hat{\rho}$, $\Omega$, $M_1$, and $M_2$ such that if
\begin{align}
E_0\le \varepsilon,
\end{align}
the problem \eqref{a1}--\eqref{a6} has a unique global strong solution $(\rho, u, d)$ in $\Omega\times(0, \infty)$
satisfying
\begin{align}
0\le \rho(x, t)\le 2\hat{\rho}, \quad (x, t)\in \Omega\times(0, \infty),
\end{align}
and for any $0<\tau<T<\infty$,
\begin{align}
\begin{cases}
(\rho, P)\in C([0, T]; W^{1, q}),\\
\nabla u\in C([\tau, T]; H^1)\cap L^2(\tau, T; W^{2, q}),\\
u_t\in L^2(\tau, T; H^1), \sqrt{\rho}u_t\in L^\infty(\tau, T; L^2),\\
\nabla d\in C([\tau, T]; H^2)\cap L^2(\tau, T; H^3), |d|=1,\\
d_t\in C([0, T];L^2)\cap L^2(\tau, T; H^2).
\end{cases}
\end{align}
\end{theorem}


\begin{remark}
Our Theorem \ref{thm1} generalizes the Cauchy problem \cite{LXZ18} to the case of bounded domains. However, this is a non-trivial generalization because we need to deal with many surface integrals caused by the boundary condition \eqref{a6}. Moreover, it should be noted that there is no need to require the compatibility condition \eqref{1.8} for the global existence of strong solutions via appropriate time-weighted techniques.
\end{remark}

\begin{remark}
If the initial data $(\rho_0,u_0,d_0)$ satisfies some additional regularity
and compatibility conditions, the global strong solutions obtained by Theorem \ref{thm1} become classical ones.
\end{remark}

\begin{remark}
Compared with \cite{CJ21}, where Cai and Li established the unique global classical solution to the compressible Navier-Stokes equations with slip boundary condition in 3D bounded domains, we cannot obtain exponential decay rates of solutions. This means that the orientation field acts some significant roles on the large-time behavior of the solutions.
\end{remark}

We now comment on the key analysis for the proof of Theorem \ref{thm1}. As emphasized in many related papers (see, e.g., \cite{CJ21,LXZ18}), the key issue is to derive the uniform-in-time lower-order estimates and uniform upper bound of the density. Compared with the Cauchy problem \cite{LXZ18}, the main difficulty lies in dealing with many surface integrals (see \eqref{z3.21}, \eqref{3.24}, and \eqref{z3.22} for example) caused by the boundary condition \eqref{a6}. To overcome this obstacle, motivated by \cite{CJ21}, we see that $(v\cdot\nabla v\cdot n)|_{\partial\Omega}=-(v\cdot\nabla n\cdot v)|_{\partial\Omega}$ for any smooth vector field $v$ satisfying $(v\cdot n)|_{\partial\Omega}=0$. Moreover, the trace theorem (see Lemma \ref{l23}) and Gagliardo-Nirenberg inequality (see Lemma \ref{l21}) as well as $L^p$-estimates based on the effective viscous flux (see \eqref{2.7} for the definition) and the vorticity (see Lemma \ref{l28}) play crucial roles. Nevertheless, compared with the compressible Navier-Stokes equations \cite{CJ21}, some additional difficulties arise due to the presence of the liquid crystal director field $d$. Especially, the supercritical nonlinearity $|\nabla d|^2d$ in the transported heat flow of harmonic map equation \eqref{a3} and the strong coupling nonlinear term $\Delta d\cdot\nabla d$ in the momentum equations \eqref{a2} will cause serious difficulties in the proofs of the time-independent estimates. To this end, we need to control the weighted $L^2$-estimate of the second-order spatial derivatives of $d$ (see Lemmas \ref{l33} and \ref{zl33}). Then we adopt the piecewise-estimate method introduced by Yu and Zhao \cite{YZ17}, which enables us to derive the weight estimates of $L^\infty(0,T)$-norm of $\|\nabla u\|_{L^2}^2+\|\Delta d\|_{L^2}^2$ and $\|\sqrt{\rho}\dot{u}\|_{L^2}^2+\|\nabla d_t\|_{L^2}^2$ step by step (see \eqref{z3.44} and \eqref{3.75}). Another key observation lies in the fact that the estimate of $\|\nabla\dot{u}\|_{L^2(t_1,t_2;L^2)}$ can be bounded by the initial energy and the factor $t_2-t_1$ for any $0\leq t_1< t_2\leq T$ (see \eqref{3.76}).
Once we overcome these difficulties, by virtue of Zlotnik's inequality (see Lemma \ref{l210}), it allows us to obtain the uniform upper bound of the density at large time provided that the initial energy is properly small (see Lemma \ref{l38}).

The rest of the paper is arranged as follows. In Section \ref{sec2}, we collect some known facts and give crucial $L^p$-estimates involving the effective viscous flux and the vorticity. In Subsection \ref{sec3.1}, we make some \textit{a priori} assumptions and show the uniformly \textit{a priori} estimates of local strong solutions independent of the time, while the energy estimates for the higher order derivatives are obtained in Subsection \ref{sec3.2}. Finally, we give the proof of Theorem \ref{thm1} in Section \ref{sec4}.

\section{Preliminaries}\label{sec2}
In this section, we recall some known facts and inequalities which will be used later.

First of all, by time-weighted techniques used in \cite{GLLZ20} and arguments as in \cite{HWW2012}, we can obtain the following local existence theorem of strong solutions of \eqref{a1}--\eqref{a6}. Here we omit the details for simplicity.
\begin{lemma}\label{l22}
Assume that the initial data $(\rho_0, u_0, d_0)$ satisfies the condition \eqref{a10}. Then there exists
a positive time $T_0>0$ and a unique strong solution $(\rho, u, d)$ of the problem \eqref{a1}--\eqref{a6} in $\Omega\times(0, T_0]$.
\end{lemma}

Next, the well-known Gagliardo-Nirenberg inequality (see \cite{N59}) will be used frequently.
\begin{lemma}\label{l21}
Assume that $\Omega$ is a bounded Lipschitz domain in $\mathbb{R}^3$. For $p\in [2, 6]$, $q\in (1, \infty)$, and $r\in (3, \infty)$,
there exist two generic constants $C_1, C_2>0$, which may depend on $p$, $q$, $r$, and $\Omega$ such that,
for any $f\in H^1(\Omega)$ and $g\in L^q(\Omega)\cap W^{1, r}(\Omega)$,
\begin{align}
&\|f\|_{L^p}\le C_1\|f\|_{L^2}^\frac{6-p}{2p}\|\nabla f\|_{L^2}^\frac{3p-6}{2p}+C_2\|f\|_{L^2},\label{2.3}\\[3pt]
&\|g\|_{L^\infty}\le C_1\|g\|_{L^q}^\frac{q(r-3)}{3r+q(r-3)}\|\nabla g\|_{L^r}^\frac{3r}{3r+q(r-3)}+C_2\|g\|_{L^2}.
\end{align}
Moreover, if $(f\cdot n)|_{\partial\Omega}=0$ and $(g\cdot n)|_{\partial\Omega}=0$, then $C_2=0$.
\end{lemma}

Next, the following trace theorem (see \cite[p. 272]{E10}) plays an important role in dealing with the boundary integral in the next section.
\begin{lemma}\label{l23}
Assume that $\Omega$ is a bounded domain and $\partial\Omega$ is $C^1$. Then there exists a bounded linear operator
\begin{align*}
T:W^{1,p}(\Omega)\rightarrow L^p(\partial\Omega),\ 1\leq p<\infty
\end{align*}
such that
\begin{align*}
Tu=u|_{\partial\Omega}\ \ \text{for}\ \ u\in W^{1,p}(\Omega)\cap C(\overline{\Omega}),
\end{align*}
and
\begin{align*}
\|Tu\|_{L^p(\partial\Omega)}\leq C\|u\|_{W^{1,p}(\Omega)}\ \ \text{for}\ \ u\in W^{1,p}(\Omega),
\end{align*}
with the constant $C$ depending only on $p$ and $\Omega$.
\end{lemma}


The following two lemmas are given in \cite{JA14, W92}.
\begin{lemma}\label{l24}
Let $k\ge 0$ be an integer and $1<r<\infty$. Assume that $\Omega$ is a simply connected bounded domain in $\mathbb{R}^3$ with
$C^{k+1, 1}$ boundary $\partial\Omega$. Then, for $v\in W^{k+1, r}(\Omega)$ with $(v\cdot n)|_{\partial\Omega}=0$, it holds that
\begin{align*}
\|v\|_{W^{k+1, r}}\le C\big(\|{\rm div}\,v\|_{W^{k, r}}+\|{\rm curl}\,v\|_{W^{k, r}}\big).
\end{align*}
In particular, for $k=0$, we have
\begin{align*}
\|\nabla v\|_{L^r}\le C\big(\|{\rm div}\,v\|_{L^r}+\|{\rm curl}\,v\|_{L^r}\big).
\end{align*}
\end{lemma}

\begin{lemma}\label{l25}
Let $k\ge 0$ be an integer and $1<r<\infty$. Suppose that $\Omega$ is bounded domain in $\mathbb{R}^3$ and its $C^{k+1, 1}$
boundary $\Omega$ has a finite number of two-dimensional connected components. Then, for $v\in W^{k+1, r}(\Omega)$ with
$(v\times n)|_{\partial\Omega}=0$, we have
\begin{align*}
\|v\|_{W^{k+1, r}}\le C\big(\|{\rm div}\,v\|_{W^{1, r}}+\|{\rm curl}\,v\|_{W^{k, r}}+\|v\|_{L^r}\big).
\end{align*}
In particular, if $\Omega$ is a simply connected bounded domain, then it holds that
\begin{align*}
\|v\|_{W^{k+1, r}}\le C\big(\|{\rm div}\,v\|_{W^{k, r}}+\|\curl v\|_{W^{k, r}}\big).
\end{align*}
\end{lemma}

When $v$ satisfies $v\cdot n=0$ on $\partial\Omega$, we will also use the identity
\begin{equation}\label{2.5}
(v\cdot\nabla)v\cdot n=-(v\cdot\nabla)n\cdot v\ \ \mbox{on}\ \ \partial\Omega
\end{equation}
for any smooth vector field $v$.

The following estimates (see \cite[Lemma 2.10]{CJ21}) on the material derivative of $u$ will be useful.
\begin{lemma}
If $(\rho, u, d)$ is a smooth solution of \eqref{a1}--\eqref{a6}. Then there exists a positive constant $C$ depending only on $\Omega$ such that
\begin{align}
\|\dot{u}\|_{L^6}& \le C\big(\|\nabla\dot{u}\|_{L^2}+\|\nabla u\|_{L^2}^2\big),\label{3.2}\\
\|\nabla\dot{u}\|_{L^2}& \le C\big(\|{\rm div}\,\dot{u}\|_{L^2}+\|{\rm curl}\,\dot{u}\|_{L^2}+\|\nabla u\|_{L^4}^2\big),
\end{align}
where $\dot{f}\triangleq f_t+u\cdot\nabla f$.
\end{lemma}

Next, we introduce the effective viscous
flux of the system \eqref{a1}--\eqref{a3} as the following
\begin{align}\label{2.7}
F\triangleq(2\mu+\lambda)\divv u-(P-\bar{P}).
\end{align}
\begin{lemma}\label{l28}
Let $(\rho, u, d)$ be a strong solution of \eqref{a1}--\eqref{a6} in $\Omega\times(0, T]$. Then, for any $p\in [2, 6]$ and $1<r<\infty$,
there exists a positive constant $C$ depending only on $p$, $r$, $\mu$, $\lambda$, and $\Omega$ such that
\begin{align}
\|\nabla u\|_{L^r}&\le C\big(\|{\rm div}\,u\|_{L^r}+\|{\rm curl}\,u\|_{L^r}\big),\label{2.8}\\
\|\nabla F\|_{L^r}&\le C\big(\|\rho\dot{u}\|_{L^r}+\||\nabla d||\nabla^2d|\|_{L^r}\big),\label{2.11}\\[3pt]
\|F\|_{L^p}&\le C\big(\|\rho\dot{u}\|_{L^2}+\||\nabla d||\nabla^2 d|\|_{L^2}\big)^\frac{3p-6}{2p}
\big(\|\nabla u\|_{L^2}+\|P-\bar{P}\|_{L^2}\big)^\frac{6-p}{2p}\notag \\
& \quad +C\big(\|\nabla u\|_{L^2}+\|P-\bar{P}\|_{L^2}\big),\\
\|{\rm curl}\,u\|_{L^p}&\le C\big(\|\rho\dot{u}\|_{L^2}+\||\nabla d||\nabla^2d|\|_{L^2}\big)^\frac{3p-6}{2p}\|\nabla u\|_{L^2}^\frac{6-p}{2p}
+C\|\nabla u\|_{L^2},\label{2.16}
\end{align}
and
\begin{align}\label{2.13}
\|\nabla \curl u\|_{L^p}\le
C\big(\|\rho\dot{u}\|_{L^p}+\||\nabla d||\nabla^2d|\|_{L^p}+\|\nabla u\|_{L^2}\big).
\end{align}
Moreover, one also has
\begin{align}
\|F\|_{L^p}+\|\curl u\|_{L^p}&\le C\big(\|\rho\dot{u}\|_{L^2}+\||\nabla d||\nabla^2d|\|_{L^2}+\|P-\bar{P}\|_{L^2}+\|\nabla u\|_{L^2}\big),
\end{align}
and
\begin{align}
\|\nabla u\|_{L^p}&\le C\big(\|\rho\dot{u}\|_{L^2}+\||\nabla d||\nabla^2d|\|_{L^2}\big)^\frac{3p-6}{2p}\big(\|\nabla u\|_{L^2}+\|P-\bar{P}\|_{L^2}\big)^\frac{6-p}{2p}\nonumber\\
&\quad+C\big(\|\nabla u\|_{L^2}+\|P-\bar{P}\|_{L^p}\big).
\end{align}

\end{lemma}
\begin{proof}[Proof]
1. Due to \eqref{a6}, we obtain \eqref{2.8} from Lemma \ref{l24}. Moreover,
by $\eqref{a2}$, $\eqref{a6}$, and \eqref{2.7},
one finds that $F$ satisfies
\begin{align}\label{2.26}
\begin{cases}
\Delta F=\divv(\rho\dot{u})+\divv\divv(M(d)) &{\rm in}~\Omega,\\
\frac{\partial F}{\partial n}=(\rho\dot{u}+\divv(M(d)))\cdot n &{\rm on}~\partial\Omega,
\end{cases}
\end{align}
where
\begin{align*}
M(d)\triangleq \nabla d\odot\nabla d-\frac12|\nabla d|^2\mathbb{I}_3,
\end{align*}
$\nabla d\odot\nabla d$ denotes a matrix whose $ij$th entry $(1\leq i,j\leq3)$
is $\partial_id\cdot\partial_jd$ and $\mathbb{I}_3$ is the $3\times3$ identity matrix.
Then we obtain \eqref{2.11} from \eqref{2.26} and \cite[Lemma 4.27]{NS04}.

2. By \eqref{2.7}, we rewrite \eqref{a2} as
\begin{align}\label{t2.16}
\mu\curl\curl u=\nabla F-\rho\dot{u}+\divv(M(d)),
\end{align}
Noticing that $\divv(\nabla\times\curl u)=0$, we get from \eqref{a6} and Lemma \ref{l25} that
\begin{align}
\|\nabla{\rm curl}\,u\|_{L^r}&\le C\big(\|\curl{\rm curl}\,u\|_{L^r}+\|{\rm curl}\,u\|_{L^r}\big)\nonumber\\
&\le C\big(\|\rho\dot{u}\|_{L^r}+\||\nabla d||\nabla^2d|\|_{L^r}+\|{\rm curl}\,u\|_{L^r}\big).\label{2.30}
\end{align}
One deduces from \eqref{2.3}, \eqref{2.11}, and \eqref{2.7} that, for $p\in [2, 6]$,
\begin{align}\label{z2.19}
\|F\|_{L^p}&\le C\|F\|_{L^2}^\frac{6-p}{2p}\|\nabla F\|_{L^2}^\frac{3p-6}{2p}+C\|F\|_{L^2}\nonumber\\
&\le C(\|\rho\dot{u}\|_{L^2}+\||\nabla d||\nabla^2d|\|_{L^2})^\frac{3p-6}{2p}
(\|\nabla u\|_{L^2}+\|P-\bar{P}\|_{L^2})^\frac{6-p}{2p}+C(\|\nabla u\|_{L^2}+\|P-\bar{P}\|_{L^2}).
\end{align}
By \eqref{2.3} and \eqref{2.30}, we arrive at
\begin{align}\label{z2.20}
\|{\rm curl}\,u\|_{L^p}&\le C\|{\rm curl}\,u\|_{L^2}^\frac{6-p}{2p}\|\nabla{\rm curl}\,u\|_{L^2}^\frac{3p-6}{2p}
+C\|{\rm curl}\,u\|_{L^2}\nonumber\\
&\le C\big(\|\rho\dot{u}\|_{L^2}+\||\nabla d||\nabla^2d|\|_{L^2}+\|\nabla u\|_{L^2}\big)^\frac{3p-6}{2p}\|\nabla u\|_{L^2}^\frac{6-p}{2p}
+C\|\nabla u\|_{L^2}\nonumber\\
&\le C\big(\|\rho\dot{u}\|_{L^2}+\||\nabla d||\nabla^2d|\|_{L^2}\big)^\frac{3p-6}{2p}\|\nabla u\|_{L^2}^\frac{6-p}{2p}
+C\|\nabla u\|_{L^2},
\end{align}
which together with \eqref{z2.19} leads to
\begin{align}
\|F\|_{L^p}+\|\curl u\|_{L^p}&\le C\big(\|\rho\dot{u}\|_{L^2}+\||\nabla d||\nabla^2d|\|_{L^2}+\|P-\bar{P}\|_{L^2}+\|\nabla u\|_{L^2}\big).
\end{align}

3. By \eqref{2.30}, \eqref{z2.20}, Young's inequality, and H\"older's inequality, we derive that, for $p\in [2, 6]$,
\begin{align}\label{2.32}
\|\nabla{\rm curl}\,u\|_{L^p}&\le C\big(\|\rho\dot{u}\|_{L^p}+\||\nabla d||\nabla^2d|\|_{L^p}+\|{\rm curl}\,u\|_{L^p}\big)\nonumber\\
&\le C\big(\|\rho\dot{u}\|_{L^p}+\||\nabla d||\nabla^2d|\|_{L^p}+\|\rho\dot{u}\|_{L^2}
+\||\nabla d||\nabla^2d|\|_{L^2}+\|\curl u\|_{L^2}\big)\nonumber\\
&\le C\big(\|\rho\dot{u}\|_{L^p}+\||\nabla d||\nabla^2d|\|_{L^p}+\|\nabla u\|_{L^2}\big),
\end{align}
which implies \eqref{2.13}.
Moreover, we infer from \eqref{2.7}, \eqref{2.8}, \eqref{z2.19}, and \eqref{z2.20} that
\begin{align*}
\|\nabla u\|_{L^p}&\le C\big(\|{\rm div}\,u\|_{L^p}+\|{\rm curl}\,u\|_{L^p}\big)\nonumber\\
&\le C\big(\|F\|_{L^p}+\|P-\bar{P}\|_{L^p}+\|\curl u\|_{L^p}\big)\nonumber\\
&\le C\big(\|\rho\dot{u}\|_{L^2}+\||\nabla d||\nabla^2d|\|_{L^2}\big)^\frac{3p-6}{2p}\big(\|\nabla u\|_{L^2}+\|P-\bar{P}\|_{L^2}\big)^\frac{6-p}{2p}\nonumber\\
&\quad+C\big(\|\nabla u\|_{L^2}+\|P-\bar{P}\|_{L^p}\big).
\end{align*}
This completes the proof.
\end{proof}

The following Beale-Kato-Majda type inequality (see \cite[Lemma 2.7]{CJ21}) will be used to estimate $\|\nabla u\|_{L^\infty}$.
\begin{lemma}\label{l211}
Let $\Omega$ be a bounded simply connected domain in $\mathbb{R}^3$ with smooth boundary. Assume that $v\in W^{2, q}(\Omega)\ (3<q<\infty)$ satisfying $v\cdot n=0$
and $\curl v\times n=0$ on $\partial\Omega$, then there exists a constant $C=C(q, \Omega)$ such that
\begin{align*}
\|\nabla v\|_{L^\infty}
\le C\big(\|\divv v\|_{L^\infty}+\|\curl v\|_{L^\infty}\big)\ln\big(e+\|\nabla^2 v\|_{L^q}\big)+C\|\nabla v\|_{L^2}+C.
\end{align*}
\end{lemma}

Finally, the following Zlotnik inequality (see \cite[Lemma 1.3]{Z00}) will be used to get the uniform-in-time upper bound of the density.
\begin{lemma}\label{l210}
Suppose the function $y$ satisfy
\begin{align*}
y'(t)=g(y)+b'(t)~on~[0, T], \quad y(0)=y^0,
\end{align*}
with $g\in C(R)$ and $y, b\in W^{1, 1}(0, T)$. If $g(\infty)=-\infty$ and
\begin{align*}
b(t_2)-b(t_1)\le N_0+N_1(t_2-t_1)
\end{align*}
for all $0\le t<t_2\le T$ with some $N_0\ge 0$ and $N_1\ge 0$, then
\begin{align*}
y(t)\le \max\{y^0, \xi_0\}+N_0<\infty~on~[0, T],
\end{align*}
where $\xi_0$ is a constant such that
\begin{align*}
g(\xi)\le -N_1, \quad for\quad \xi\ge \xi_0.
\end{align*}
\end{lemma}

\section{\textit{A priori} estimates}\label{sec3}
In this section, we will establish some necessary \textit{a priori} bounds for strong solutions to the problem \eqref{a1}--\eqref{a6} in order to extend the local strong solutions guaranteed by Lemma \ref{l22}. Let $T>0$ be a fixed time and $(\rho, u, d)$ be a strong solution to \eqref{a1}--\eqref{a6} in $\Omega\times (0, T]$ with initial data $(\rho_0, u_0, d_0)$ satisfying \eqref{a10}.

\subsection{Lower-order estimates}\label{sec3.1}
Throughout this subsection, we will use $C$ or $C_i\ (i=1, 2, \cdots)$
to denote the generic positive constants, which may depend on $\mu$, $\lambda$, $\gamma$, $a$, $\hat{\rho}$, $\Omega$, $M_1$, $M_2$, and $\bar{\rho}$. In particular, they are independent of $T$. Sometimes we use $C(\alpha)$ to emphasize the dependence of $C$ on $\alpha$.

Set $\sigma=\sigma(t)\triangleq\min\{1, t\}$, we define
\begin{align}\label{3.5}
A_1(T)\triangleq\sup_{0\le t\le T}\big[\sigma(t)\big(\|\nabla u\|_{L^2}^2+\|\Delta d\|_{L^2}^2\big)\big],
\quad A_2(T)\triangleq\sup_{0\le t\le T}\big(\|\nabla u\|_{L^2}^2+\|\Delta d\|_{L^2}^2\big).
\end{align}
The main aim of this subsection is to obtain the following key \textit{a priori} estimates, which give the uniform upper bound of the density.
\begin{proposition}\label{p31}
Under the conditions of Theorem \ref{thm1}, there exist positive constants
$\varepsilon$ and $K$ both depending on $\mu$, $\lambda$,
$\gamma$, $a$, $\bar{\rho}$, $\hat{\rho}$, $\Omega$, $M_1$, and $M_2$ such that if $(\rho, u, d)$ is a strong solution of \eqref{a1}--\eqref{a6} in
$\Omega\times (0, T]$ satisfying
\begin{align}\label{3.6}
\sup_{\Omega\times[0, T]}\rho\le 2\hat{\rho}, \quad A_1(T)\le 2E_0^\frac12, \quad A_2(\sigma(T))\le 4K,
\end{align}
then the following estimates hold
\begin{align}\label{3.7}
\sup_{\Omega\times[0, T]}\rho\le \frac74\hat{\rho}, \quad A_1(T)\le E_0^\frac12, \quad A_2(\sigma(T))\le 3K,
\end{align}
provided that $E_0\le\varepsilon$.
\end{proposition}
\begin{remark}
Recalling the definition of $\sigma(t)$, we then obtain from \eqref{3.6} that
\begin{align}\label{t3.4}
\sup_{0\le t\le T}\big(\|\nabla u\|_{L^2}^2+\|\Delta d\|_{L^2}^2\big)\le C.
\end{align}
\end{remark}

Before proving Proposition \ref{p31}, we show some necessary \textit{a priori} estimates, see Lemmas \ref{l32}--\ref{l38} below.

\begin{lemma}\label{l32}
Let $(\rho, u, d)$ be a strong solution of \eqref{a1}--\eqref{a6} in $\Omega\times [0, T]$, then it holds that
\begin{align}
&\sup_{0\le t\le T}\Big(\frac12\|\sqrt{\rho}u\|_{L^2}^2+\frac12\|\nabla d\|_{L^2}^2+\|G(\rho)\|_{L^1}\Big)\nonumber\\
&\quad+\int_0^T\big[\mu\|\nabla u\|_{L^2}^2+(\mu+\lambda)\|\divv u\|_{L^2}^2+\|\Delta d+|\nabla d|^2d\|_{L^2}^2\big]dt\le E_0,\label{3.8}\\
&\sup_{0\le t\le T}\big(\|d_t\|_{L^2}^2+\|\nabla^2d\|_{L^2}^2\big)\le C.\label{t3.6}
\end{align}
Moreover, for any integer $1\le i\le [T]-1$, one has
\begin{align}\label{3.9}
\sup_{0\le t\le T}\|\rho-\bar{\rho}\|_{L^2}^2
+\int_{i-1}^{i+1}\big(\|d_t\|_{L^2}^2+\|\nabla^2d\|_{L^2}^2\big)dt\le CE_0^{\frac12},
\end{align}
provided that $E_0\le 1$.
\end{lemma}
\begin{proof}[Proof]
1. Due to
\begin{align}\label{t3.8}
-\Delta u=-\nabla{\rm div}\,u+\curl\curl u,
\end{align}
we rewrite $\eqref{a2}$ as
\begin{align}\label{3.10}
\rho u_t+\rho u\cdot\nabla u-(2\mu+\lambda)\nabla{\rm div}\,u+\mu\curl\curl u
+\nabla(P-\bar{P})+\Delta d\cdot\nabla d=0.
\end{align}
Multiplying \eqref{3.10} by $u$ and \eqref{a1} by $G'(\rho)$, respectively, summing up, and integrating the resulting equality over $\Omega$, we get that
\begin{align}\label{3.12}
&\frac{d}{dt}\int\Big(\frac12\rho|u|^2
+G(\rho)\Big)dx+(2\mu+\lambda)\int({\rm div}\,u)^2dx+\mu\int|{\rm curl}\,u|^2dx\nonumber\\
&=-\int\divv(\rho u)G'(\rho)dx-\int u\cdot\nabla(P-\bar{P})dx-\int u \cdot \nabla d\cdot\Delta ddx\nonumber\\
&=\int \rho u\cdot\nabla Q(\rho)dx-\int u\cdot\nabla Pdx-\int u \cdot \nabla d\cdot\Delta d dx\nonumber\\
&=-\int u \cdot\nabla d\cdot\Delta d dx,
\end{align}
where we have used \eqref{t3.8} and \eqref{a6} to obtain
\begin{align*}
\int\mathcal Lu\cdot udx&=\int\big[(2\mu+\lambda)\nabla\divv u\cdot u-\mu\curl\curl u\cdot u\big]dx\nonumber\\
&=-\int\big[(2\mu+\lambda)(\divv u)^2+\mu|\curl u|^2\big]dx,
\end{align*}
and
\begin{align*}
G'(\rho)=Q(\rho)-Q(\bar{\rho}),\quad Q'(\rho)=P'(\rho)/\rho.
\end{align*}
Multiplying \eqref{a3} by $\Delta d+|\nabla d|^2d$ and integrating by parts, we derive after using $|d|=1$ and $\frac{\partial d}{\partial n}|_{\partial\Omega}=0$
that
\begin{align*}
&\frac12\frac{d}{dt}\int|\nabla d|^2dx+\int|\Delta d+|\nabla d|^2d|^2dx\\
&=\int u\cdot\nabla d\cdot\Delta ddx
+\int\big(|\nabla d|^2 d\cdot d_t+|\nabla d|^2(u\cdot\nabla) d\cdot d\big)dx\\
&=\int u\cdot\nabla d\cdot\Delta ddx
+\frac12\int\big(|\nabla d|^2 \partial_t|d|^2+|\nabla d|^2u\cdot\nabla|d|^2\big)dx\\
&=\int u\cdot\nabla d\cdot\Delta ddx,
\end{align*}
which together with \eqref{3.12} yields \eqref{3.8}.

2. Integration by parts, we deduce from \eqref{a6}, \eqref{2.5}, and Lemma \ref{l23} that
\begin{align}\label{lz1}
\|\Delta d\|_{L^2}^2&=\sum_{i, j=1}^3\int\partial_{ii}d\cdot\partial_{jj}ddx=
-\sum_{i, j=1}^3\int\partial_id\cdot\partial_i\partial_{jj}ddx\nonumber\\
&=\sum_{i, j=1}^3\int|\partial_{ij}d|^2dx
-\sum_{i, j=1}^3\int_{\partial\Omega}\partial_id\cdot\partial_{ij}dn^jdS\nonumber\\
&=\sum_{i, j=1}^3\int|\partial_{ij}d|^2dx+\sum_{i, j=1}^3\int_{\partial\Omega}\partial_id\partial_in^j\partial_jddS\nonumber\\
&\ge \|\nabla^2 d\|_{L^2}^2-C\||\nabla d|^2\|_{W^{1, 1}}\nonumber\\
&\ge \|\nabla^2 d\|_{L^2}^2-C\|\nabla d\|_{L^2}\|\nabla^2d\|_{L^2}-C\|\nabla d\|_{L^2}^2\nonumber\\
&\ge \frac12\|\nabla^2 d\|_{L^2}^2-C\|\nabla d\|_{L^2}^2,
\end{align}
which combined with \eqref{t3.4} and \eqref{3.8} implies that
\begin{align}\label{t3.15}
\sup_{0\le t\le T}\|\nabla^2 d\|_{L^2}^2\le C.
\end{align}
It follows from \eqref{a3}, \eqref{2.3}, and \eqref{a6} that
\begin{align*}
\|d_t\|_{L^2}^2&\le C\big(\||u||\nabla d|\|_{L^2}^2+\||\nabla d|^2\|_{L^2}^2+\|\Delta d\|_{L^2}^2\big)\nonumber\\
&\le C\big(\|u\|_{L^6}^2\|\nabla d\|_{L^3}^2+\|\nabla d\|_{L^4}^4+\|\Delta d\|_{L^2}^2\big)\nonumber\\
&\le C\|\nabla u\|_{L^2}^2\|\nabla d\|_{L^2}\|\nabla^2d\|_{L^2}
+C\|\nabla d\|_{L^2}\|\nabla^2d\|_{L^2}^3+C\|\Delta d\|_{L^2}^2,
\end{align*}
which together with \eqref{3.8} and \eqref{t3.15} leads to \eqref{t3.6}.

3. In view of \eqref{1.7}, we see that there exists a positive constant $C$ depending only on $a$, $\gamma$, and $\hat{\rho}$ such that
\begin{align*}
&|P-\bar{P}|\le C|\rho-\bar{\rho}|, \quad C^{-1}(\rho-\bar{\rho})^2\le G(\rho)\le C(\rho-\bar{\rho})^2,
\end{align*}
which along with \eqref{3.8} gives that
\begin{align}\label{t3.13}
\sup_{0\le t\le T}\|\rho-\bar{\rho}\|_{L^2}^2\le CE_0.
\end{align}
We derive from \eqref{a3}, \eqref{a6}, \eqref{2.3}, \eqref{t3.4}, and \eqref{t3.15} that
\begin{align*}
\frac{d}{dt}\|\nabla d\|_{L^2}^2+\|d_t\|_{L^2}^2+\|\Delta d\|_{L^2}^2
&=\int|d_t-\Delta d|^2dx \notag \\
&=\int|u\cdot\nabla d-|\nabla d|^2|^2dx\nonumber\\
&\le C\|u\|_{L^6}^2\|\nabla d\|_{L^3}^2+C\|\nabla d\|_{L^4}^4\nonumber\\
&\le C\|\nabla u\|_{L^2}^2\|\nabla d\|_{L^2}\|\nabla^2d\|_{L^2}
+C\|\nabla d\|_{L^2}\|\nabla^2d\|_{L^2}^3\nonumber\\
&\le CE_0^\frac12,
\end{align*}
which together with \eqref{lz1} and \eqref{3.8} leads to
\begin{align}\label{t3.16}
\frac{d}{dt}\|\nabla d\|_{L^2}^2+\|d_t\|_{L^2}^2+\|\nabla^2 d\|_{L^2}^2
\le CE_0^{\frac12},
\end{align}
provided that $E_0\le 1$.
Integrating \eqref{t3.16} over $[0, \sigma(T)]$ and using \eqref{3.8}, we have
\begin{align}
\int_0^{\sigma(T)}\big(\|d_t\|_{L^2}^2+\|\nabla^2d|_{L^2}^2\big)dt\le CE_0^{\frac12}.
\end{align}
Denote $\sigma_i\triangleq\sigma(t+1-i)$. For any integer $1\le i\le [T]-1$, multiplying \eqref{t3.16} by $\sigma_i$, we arrive at
\begin{align}\label{t3.18}
\frac{d}{dt}\big(\sigma_i\|\nabla d\|_{L^2}^2\big)+\sigma_i\big(\|d_t\|_{L^2}^2+\|\nabla^2d\|_{L^2}^2\big)
\le \sigma_i'\|\nabla d\|_{L^2}^2+CE_0^{\frac12}\sigma_i
\le \|\nabla d\|_{L^2}^2+CE_0^{\frac12}.
\end{align}
Integrating \eqref{t3.18} over $(i-1, i+1]$, we obtain \eqref{3.9} from \eqref{3.8} and \eqref{t3.13}.
\end{proof}

\begin{lemma}\label{l33}
Let $(\rho, u, d)$ be a strong solution
of \eqref{a1}--\eqref{a6} satisfying \eqref{3.6}. Assume that
$\eta(t)\ge 0$ is a piecewise differentiable function, then it holds that
\begin{align}\label{t3.20}
&\frac{d}{dt}\Big(\frac{2\mu+\lambda}{2}\eta(t)\|{\rm div}\,u\|_{L^2}^2
+\frac{\mu}{2}\eta(t)\|{\rm curl}\,u\|_{L^2}^2+\eta(t)\|\Delta d\|_{L^2}^2\Big)+\frac12\eta(t)\|\sqrt{\rho}\dot{u}\|_{L^2}^2
+\frac12\eta(t)\|\nabla d_t\|_{L^2}^2\nonumber\\
&\le \frac{d}{dt}\int\eta(t)(P-\bar{P}){\rm div}\,udx+\frac{d}{dt}\int\eta(t)M(d):\nabla udx
+C\big(\eta(t)+|\eta'(t)|\big)\|\nabla u\|_{L^2}^2
\nonumber\\
&\quad+C\eta(t)\big(\|\nabla u\|_{L^3}^3+\|\nabla u\|_{L^2}^4
+\|\nabla u\|_{L^2}^2+\|\nabla d\|_{H^1}^2+\|\nabla d\|_{H^1}^6
+\|\nabla u\|_{L^2}^4\|\nabla d\|_{H^1}^2\big)\nonumber\\
&\quad +C|\eta'(t)\|\Delta d\|_{L^2}^2+C|\eta'(t)|\big(\|\nabla d\|_{L^2}\|\nabla d\|_{H^1}^3+\|\nabla u\|_{L^2}^2\big)
+C\eta(t)\|\nabla u\|_{L^2}^2\|\nabla d\|_{H^1}^2+C|\eta'(t)|E_0,
\end{align}
provided that $E_0\le \varepsilon_2$.
\end{lemma}
\begin{proof}[Proof]
1. Multiplying $\eqref{3.10}$ by $\eta(t)\dot{u}$ and integrating the resulting equality over $\Omega$ lead to
\begin{align}\label{3.17}
\int\eta(t)\rho|\dot{u}|^2dx
&=-\int\eta(t)\dot{u}\cdot\nabla(P-\bar{P})dx+(2\mu+\lambda)\int\eta(t)\nabla\divv u\cdot\dot{u}dx\nonumber\\
&\quad-\mu\int\eta(t)\curl\curl u\cdot\dot{u}dx
-\int\eta(t)\dot{u}\cdot\Delta d\cdot\nabla ddx
\triangleq\sum_{i=1}^4I_i.
\end{align}
By $\eqref{a1}_1$ and $P=a\rho^\gamma$, we have
\begin{align}\label{z3.13}
P_t+\divv(Pu)+(\gamma-1)P\divv u=0,
\end{align}
which together with integration by parts and \eqref{3.6} shows that
\begin{align}\label{3.20}
I_1&=-\int\eta(t)u_t\cdot\nabla(P-\bar{P})dx-\int\eta(t)u\cdot\nabla u\cdot\nabla Pdx\nonumber\\
&=\frac{d}{dt}\int\eta(t)(P-\bar{P}){\rm div}\,udx-\eta'(t)\int(P-\bar{P}){\rm div}\,udx-\int\eta(t){\rm div}\,uP_tdx
-\int\eta(t)u\cdot\nabla u\cdot\nabla Pdx\nonumber\\
&=\frac{d}{dt}\int\eta(t)(P-\bar{P}){\rm div}\,udx-\eta'(t)\int(P-\bar{P}){\rm div}\,udx+\int\eta(t){\rm div}\,u{\rm div}\,(Pu)dx\nonumber\\
&\quad+(\gamma-1)\int\eta(t)P({\rm div}\,u)^2dx-\int\eta(t)u\cdot\nabla u\cdot\nabla Pdx\nonumber\\
&=\frac{d}{dt}\int\eta(t)(P-\bar{P}){\rm div}\,udx-\eta'(t)\int(P-\bar{P}){\rm div}\,udx
+\int\eta(t)P\nabla u:\nabla udx\nonumber\\
&\quad+(\gamma-1)\int\eta(t)P({\rm div}\,u)^2dx-\int_{\partial\Omega}\eta(t)Pu\cdot\nabla u\cdot ndS
\nonumber\\
&\le \frac{d}{dt}\int\eta(t)(P-\bar{P}){\rm div}\,udx+C\eta(t)\|\nabla u\|_{L^2}^2
+|\eta'(t)|\|P-\bar{P}\|_{L^2}\|\nabla u\|_{L^2}\nonumber\\
&\le \frac{d}{dt}\int\eta(t)(P-\bar{P}){\rm div}\,udx+C(\eta(t)+|\eta'(t)|)\|\nabla u\|_{L^2}^2+C|\eta'(t)|\|P-\bar{P}\|_{L^2}^2\nonumber\\
&\le \frac{d}{dt}\int\eta(t)(P-\bar{P}){\rm div}\,udx+C(\eta(t)+|\eta'(t)|)\|\nabla u\|_{L^2}^2+C|\eta'(t)|E_0,
\end{align}
where we have used
\begin{align*}
\int\eta(t){\rm div}\,(Pu){\rm div}\,udx
&=-\int\eta(t) Pu^j\partial_{ji}u^idx
=\int\eta(t)\partial_i(Pu^j)\partial_ju^idx\nonumber\\
&=-\int_{\partial\Omega}\eta(t)Pu\cdot\nabla u\cdot ndS
+\int\eta(t)\partial_iPu^j\partial_ju^idx+\int\eta(t)P\partial_iu^j\partial_ju^idx\nonumber\\
&=-\int_{\partial\Omega}\eta(t)Pu\cdot\nabla u\cdot ndS+\int\eta(t) u\cdot\nabla u\cdot\nabla Pdx+\int\eta(t) P\nabla u:\nabla udx,
\end{align*}
and
\begin{align}\label{z3.21}
-\int_{\partial\Omega}\eta(t)Pu\cdot\nabla u\cdot ndS&=\int_{\partial\Omega}\eta(t)Pu\cdot\nabla n\cdot udS
\le C\eta(t)\int_{\partial\Omega}|u|^2dS\le C\eta(t)\|\nabla u\|_{L^2}^2,
\end{align}
due to \eqref{2.5}, \eqref{3.6}, Lemma \ref{l23}, and \eqref{2.3}. Here and in what follows, we use the Einstein convention that the repeated indices denote the summation.

2. By \eqref{a6} and \eqref{2.5}, we derive from integration by parts that
\begin{align*}
I_2&=(2\mu+\lambda)\int_{\partial\Omega}\eta(t){\rm div}\,u(\dot{u}\cdot n)dS-(2\mu+\lambda)\int\eta(t){\rm div}\,u{\rm div}\,\dot{u}dx\nonumber\\
&=(2\mu+\lambda)\int_{\partial\Omega}\eta(t){\rm div}\,u(u\cdot\nabla u\cdot n)dS-\frac{2\mu+\lambda}{2}\frac{d}{dt}\int\eta(t)({\rm div}\,u)^2dx\nonumber\\
&\quad-(2\mu+\lambda)\int\eta(t){\rm div}\,u{\rm div}\,(u\cdot\nabla u)dx+\frac{2\mu+\lambda}{2}\eta'(t)\int({\rm div}u)^2dx\nonumber\\
&=-\frac{2\mu+\lambda}{2}\frac{d}{dt}\int\eta(t)(\divv u)^2dx
-(2\mu+\lambda)\int_{\partial\Omega}\eta(t)\divv u(u\cdot\nabla n\cdot u)dS\nonumber\\
&\quad-(2\mu+\lambda)\int\eta(t){\rm div}\,u\partial_i(u^j\partial_ju^i)dx+\frac{2\mu+\lambda}{2}\eta'(t)\int({\rm div}\,u)^2dx\nonumber\\
&=-\frac{2\mu+\lambda}{2}\frac{d}{dt}\int\eta(t)({\rm div}\,u)^2dx
-(2\mu+\lambda)\int_{\partial\Omega}\eta(t){\rm div}\,u(u\cdot\nabla n\cdot u)dS\nonumber\\
&\quad-(2\mu+\lambda)\int\eta(t){\rm div}\,u\nabla u:\nabla udx-(2\mu+\lambda)\int\eta(t){\rm div}\,uu^j\partial_{ji}u^idx+\frac{2\mu+\lambda}{2}\eta'(t)\int({\rm div}\,u)^2dx\nonumber\\
&=-\frac{2\mu+\lambda}{2}\frac{d}{dt}\int\eta(t)({\rm div}\,u)^2dx
-(2\mu+\lambda)\int_{\partial\Omega}\eta(t){\rm div}\,u(u\cdot\nabla n\cdot u)dS\nonumber\\
&\quad-(2\mu+\lambda)\int\eta(t){\rm div}\,u\nabla u:\nabla udx+\frac{2\mu+\lambda}{2}\int\eta(t)({\rm div}\,u)^3dx
+\frac{2\mu+\lambda}{2}\eta'(t)\int({\rm div}\,u)^2dx\nonumber\\
&\le -\frac{2\mu+\lambda}{2}\frac{d}{dt}\int\eta(t)({\rm div}\,u)^2dx
+\frac12\eta(t)\|\sqrt{\rho}\dot{u}\|_{L^2}^2+\delta\eta(t)\|\nabla^3d\|_{L^2}^2
+C|\eta'(t)|\|\nabla u\|_{L^2}^2\nonumber\\
&\quad+C\eta(t)\big(\|\nabla u\|_{L^3}^3+\|\nabla u\|_{L^2}^4
+\|\nabla u\|_{L^2}^2+\|\nabla d\|_{H^1}^2\big),
\end{align*}
where we have used
\begin{align*}
\int\eta(t){\rm div}\,uu^j\partial_{ji}u^idx
&=-\int\eta(t)\partial_j(\partial_ku^ku^j)\partial_iu^idx\nonumber\\
&=-\int\eta(t)\partial_{jk}u^ku^j\partial_iu^idx
-\int\eta(t)\divv u\partial_ju^j\partial_iu^idx\nonumber\\
&=-\int\eta(t)\partial_{ji}u^iu^j\divv udx
-\int\eta(t)({\rm div}\,u)^3dx,
\end{align*}
and
\begin{align}\label{3.24}
&\Big|-(2\mu+\lambda)\int_{\partial\Omega}{\rm div}\,u(u\cdot\nabla n\cdot u)dS\Big|\nonumber\\
&=\Big|-\int_{\partial\Omega}\big(F+(P-\bar{P})\big)(u\cdot\nabla n\cdot u)dS\Big|\nonumber\\
&\le \Big|\int_{\partial\Omega}F(u\cdot\nabla n\cdot u)dS\Big|+\Big|\int_{\partial\Omega}(P-\bar{P})(u\cdot\nabla n\cdot u)dS\Big|
\nonumber\\
&\le C\int_{\partial\Omega}|F_1||u|^2dS+C\int_{\partial\Omega}|u|^2dS\nonumber\\
&\le C\big(\|\nabla F\|_{L^2}\|u\|_{L^4}^2+\|F\|_{L^6}\|u\|_{L^3}\|\nabla u\|_{L^2}
+\|F\|_{L^2}\|u\|_{L^4}^2\big)+C\|\nabla u\|_{L^2}^2\nonumber\\
&\le C\|F\|_{H^1}\|u\|_{H^1}^2+C\|\nabla u\|_{L^2}^2\nonumber\\
&\le \frac{1}{2}\|\sqrt{\rho}\dot{u}\|_{L^2}^2+\delta\|\nabla^3 d\|_{L^2}^2
+C\big(\|\nabla u\|_{L^2}^4+\|\nabla u\|_{L^2}^2+\|\nabla d\|_{H^1}^2\big),
\end{align}
due to Lemma \ref{l23}, \eqref{2.3}, Lemma \ref{l28}, and
\begin{align*}
\||\nabla d||\nabla^2 d|\|_{L^2}
&\le C\|\nabla d\|_{L^6}\|\nabla^2d\|_{L^3} \\
&\le C\|\nabla d\|_{H^1}\|\nabla^2d\|_{L^2}^\frac12\|\nabla^2d\|_{L^6}^\frac12 \\
& \le C\|\nabla d\|_{H^1}^\frac32\|\nabla^2 d\|_{L^6}^\frac12\  \\
&\le C\|\nabla d\|_{H^1}^\frac32\big(\|\nabla^3d\|_{L^2}+\|\nabla^2d\|_{L^2}\big)^\frac12\ \ (\text{by Lemma } \ref{l25}) \\
&\le C\|\nabla d\|_{H^1}^\frac32\|\nabla^3d\|_{L^2}^\frac12+\|\nabla d\|_{H^1}^2.
\end{align*}

3. Noting that
\begin{align*}
\int{\rm curl}\,u\cdot(u^i\partial_i{\rm curl}\,u)dx=-\int{\rm curl}\,u\cdot(u^i\partial_i{\rm curl}\,u)dx-\int|{\rm curl}\,u|^2{\rm div}\,udx.
\end{align*}
Thus, we have
\begin{align*}
\int{\rm curl}\,u\cdot(u^i\partial_i{\rm curl}\,u)dx
=-\frac12\int|{\rm curl}\,u|^2{\rm div}\,udx.
\end{align*}
This implies that
\begin{align*}
\mu\int{\rm curl}\,u\cdot{\rm curl}\,(u\cdot\nabla u)dx
&=\mu\int{\rm curl}\,u\cdot{\rm curl}\,(u^i\partial_i u)dx\nonumber\\
&=\mu\int{\rm curl}\,u\cdot\big(u^i{\rm curl}\,\partial_iu+\nabla u^i\times\partial_iu\big)dx\nonumber\\
&=-\mu\int\partial_i({\rm curl}\,u u^i){\rm curl}\,udx+\mu\int(\nabla u^i\times\partial_iu)\cdot{\rm curl}\,udx\nonumber\\
&=-\frac{\mu}{2}\int|{\rm curl}\,u|^2{\rm div}\,udx+\mu\int(\nabla u^i\times\partial_iu)\cdot{\rm curl}\,udx,
\end{align*}
which combined with \eqref{a6} and integration by parts leads to
\begin{align}\label{3.26}
I_3&=-\mu\int\eta(t){\rm curl}\,u\cdot{\rm curl}\,\dot{u}dx\nonumber\\
&=-\frac{\mu}{2}\frac{d}{dt}\int\eta(t)|{\rm curl}\,u|^2dx
+\frac{\mu}{2}\eta'(t)\int|{\rm curl}\,u|^2dx
-\mu\int\eta(t){\rm curl}\,u\cdot{\rm curl}\,(u\cdot\nabla u)dx\nonumber\\
&=-\frac{\mu}{2}\frac{d}{dt}\int\eta(t)|{\rm curl}\,u|^2dx
+\frac{\mu}{2}\eta'(t)\int|{\rm curl}\,u|^2dx\nonumber\\
&\quad-\mu\int\eta(t)(\nabla u^i\times\partial_iu)\cdot{\rm curl}\,udx
+\frac{\mu}{2}\int\eta(t)|{\rm curl}\,u|^2{\rm div}\,udx\nonumber\\
&\le -\frac{\mu}{2}\frac{d}{dt}\int\eta(t)|{\rm curl}\,u|^2dx
+C|\eta'(t)|\|\nabla u\|_{L^2}^2+C\eta(t)\|\nabla u\|_{L^3}^3.
\end{align}

4. Noticing that
\begin{align*}
I_4&=-\int\eta(t)u_t\cdot\Delta d\cdot\nabla ddx-\int\eta(t) u\cdot\nabla u\cdot\Delta d\cdot\nabla ddx
\triangleq I_{41}+I_{42}.
\end{align*}
Using \eqref{a6}, H\"older's inequality, Sobolev's inequality, and \eqref{2.3}, we have
\begin{align}\label{t3.27}
I_{41}&=\int\eta(t)M(d):\nabla u_tdx\nonumber\\
&=\frac{d}{dt}\int\eta(t)M(d):\nabla udx-\eta'(t)\int M(d):\nabla udx-\int\eta(t)M(d)_t:\nabla udx\nonumber\\
&\le\frac{d}{dt}\int\eta(t)M(d):\nabla udx+C|\eta'(t)|\|\nabla d\|_{L^4}^2\|\nabla u\|_{L^2}+C\eta(t)\|\nabla u\|_{L^3}\|\nabla d_t\|_{L^2}\|\nabla d\|_{L^6}\nonumber\\
&\le \frac{d}{dt}\int\eta(t)M(d):\nabla udx+C|\eta'(t)|\big(\|\nabla d\|_{L^2}\|\nabla d\|_{H^1}^3+\|\nabla u\|_{L^2}^2\big)\nonumber\\
&\quad+\delta\eta(t)\|\nabla d_t\|_{L^2}^2+C\eta(t)\big(\|\nabla u\|_{L^3}^3+\|\nabla d\|_{H^1}^6\big),
\end{align}
and
\begin{align}\label{t3.28}
I_{42}&\le C\eta(t)\|u\|_{L^6}\|\nabla u\|_{L^2}\|\Delta d\|_{L^6}\|\nabla d\|_{L^6}\nonumber\\
&\le C\eta(t)\|\nabla u\|_{L^2}^2\big(\|\nabla^3 d\|_{L^2}+\|\nabla^2d\|_{L^2}\big)\|\nabla d\|_{H^1}\nonumber\\
&\le \delta\eta(t)\|\nabla^3d\|_{L^2}^2+C\eta(t)\|\nabla u\|_{L^2}^4\|\nabla d\|_{H^1}^2+C\eta(t)\|\nabla u\|_{L^2}^2\|\nabla d\|_{H^1}^2.
\end{align}
Combining \eqref{t3.27} and \eqref{t3.28}, we deduce that
\begin{align*}
I_4&\le \frac{d}{dt}\int\eta(t)M(d):\nabla udx+\delta\eta(t)\big(\|\nabla d_t\|_{L^2}^2+\|\nabla^3d\|_{L^2}^2\big)
+C|\eta'(t)|\big(\|\nabla d\|_{L^2}\|\nabla d\|_{H^1}^3+\|\nabla u\|_{L^2}^2\big)\nonumber\\
&\quad+C\eta(t)\big(\|\nabla u\|_{L^3}^3+\|\nabla d\|_{H^1}^6
+\|\nabla u\|_{L^2}^4\|\nabla d\|_{H^1}^2+\|\nabla u\|_{L^2}^2\|\nabla d\|_{H^1}^2\big).
\end{align*}
Putting the above estimates on $I_i\ (i=1, 2, 3, 4)$ into \eqref{3.17}, one obtains that
\begin{align}\label{2.35}
&\frac{d}{dt}\Big(\frac{2\mu+\lambda}{2}\eta(t)\|{\rm div}\,u\|_{L^2}^2
+\frac{\mu}{2}\eta(t)\|{\rm curl}\,u\|_{L^2}^2\Big)+\frac12\eta(t)\|\sqrt{\rho}\dot{u}\|_{L^2}^2\nonumber\\
&\le \frac{d}{dt}\int\eta(t)(P-\bar{P}){\rm div}\,udx+\frac{d}{dt}\int\eta(t)M(d):\nabla udx
+\delta\eta(t)\big(\|\nabla^3d\|_{L^2}^2+\|\nabla d_t\|_{L^2}^2\big)\nonumber\\
&\quad+C\big(\eta(t)+|\eta'(t)|\big)\|\nabla u\|_{L^2}^2+C|\eta'(t)|\big(\|\nabla d\|_{L^2}\|\nabla d\|_{H^1}^3+\|\nabla u\|_{L^2}^2\big)\nonumber\\
&\quad+C\eta(t)\big(\|\nabla u\|_{L^3}^3+\|\nabla u\|_{L^2}^4
+\|\nabla u\|_{L^2}^2+\|\nabla d\|_{H^1}^2+\|\nabla d\|_{H^1}^6
+\|\nabla u\|_{L^2}^4\|\nabla d\|_{H^1}^2\big)\nonumber\\
&\quad+C\eta(t)\|\nabla u\|_{L^2}^2\|\nabla d\|_{H^1}^2+C|\eta'(t)|E_0.
\end{align}

5. It remain to estimate $\|\nabla d_t\|_{L^2}^2$. To this end, applying the operator $\nabla$ to \eqref{a3} gives that
\begin{align}\label{z3.31}
\nabla d_t-\nabla\Delta d=-\nabla(u\cdot\nabla d)+\nabla(|\nabla d|^2d).
\end{align}
Multiplying \eqref{z3.31} by $\nabla d_t$ and integration by parts, we find that
\begin{align*}
&\frac{d}{dt}\int|\Delta d|^2dx+\int|\nabla d_t|^2dx\nonumber\\
&=\int\big(\nabla(|\nabla d|^2d)-\nabla(u\cdot\nabla d)\big)\nabla d_tdx\nonumber\\
&\le \frac14\|\nabla d_t\|_{L^2}^2+C\int\big(|\nabla d|^2|\nabla^2d|^2+|\nabla d|^6+|\nabla u|^2|\nabla d|^2+|u|^2|\nabla^2d|^2\big)dx\nonumber\\
&\le \delta\|\nabla d_t\|_{L^2}^2+C\|\nabla d\|_{L^3}^2\|\nabla^2d\|_{L^6}^2+C\|\nabla d\|_{L^6}^6
+C\|\nabla u\|_{L^3}^2\|\nabla d\|_{L^6}^2
+C\|u\|_{L^6}^2\|\nabla^2d\|_{L^3}^2\nonumber\\
&\le \frac14\|\nabla d_t\|_{L^2}^2+C\|\nabla d\|_{L^2}\|\nabla d\|_{L^6}\big(\|\nabla^3d\|_{L^2}^2+\|\nabla^2d\|_{L^2}^2\big)
+C\|\nabla d\|_{H^1}^6\nonumber\\
&\quad+C\|\nabla u\|_{L^3}^3+C\|\nabla u\|_{L^2}^2\|\nabla^2d\|_{L^2}\big(\|\nabla^3d\|_{L^2}+\|\nabla^2d\|_{L^2}\big)\nonumber\\
&\le \frac14\|\nabla d_t\|_{L^2}^2+\Big(CE_0^\frac12+\delta\Big)\|\nabla^3d\|_{L^2}^2
+C\|\nabla u\|_{L^2}^2\|\nabla^2d\|_{L^2}^2
+C\|\nabla u\|_{L^2}^4\|\nabla^2d\|_{L^2}^2 \notag \\
& \quad +C\|\nabla u\|_{L^3}^3+C\|\nabla d\|_{H^1}^2+C\|\nabla d\|_{H^1}^6,
\end{align*}
which implies that
\begin{align}\label{t3.33}
&\frac{d}{dt}(\eta(t)\|\Delta d\|_{L^2}^2)+\eta(t)\|\nabla d_t\|_{L^2}^2\nonumber\\
&\le \eta'(t)\|\Delta d\|_{L^2}^2+\frac14\eta(t)\|\nabla d_t\|_{L^2}^2+\Big(CE_0^\frac12+\delta\Big)\eta(t)\|\nabla^3d\|_{L^2}^2
+C\eta(t)\|\nabla d\|_{H^1}^4\nonumber\\
&\quad+C\eta(t)\|\nabla u\|_{L^2}^2\|\nabla^2d\|_{L^2}^2+C\eta(t)\|\nabla u\|_{L^2}^4\|\nabla^2d\|_{L^2}^2\notag \\
& \quad +C\eta(t)\|\nabla u\|_{L^3}^3+C\eta(t)\|\nabla d\|_{H^1}^2+C\eta(t)\|\nabla d\|_{H^1}^6.
\end{align}
Applying the $L^2$-theory to the Neumann boundary value problem of elliptic equations (see \cite{L13}),
we infer from \eqref{z3.31}, \eqref{3.8}, and \eqref{t3.6} that
\begin{align*}
\|\nabla^3d\|_{L^2}^2&\le C\|\nabla\Delta d\|_{L^2}^2+\|\nabla d\|_{H^1}^2\nonumber\\
&\le C\|\nabla d_t\|_{L^2}^2+C\|\nabla(u\cdot\nabla d)\|_{L^2}^2
+C\|\nabla(|\nabla d|^2d)\|_{L^2}^2+C\|\nabla d\|_{H^1}^2\nonumber\\
&\le C\|\nabla d_t\|_{L^2}^2+\Big(CE_0^\frac12+\frac14\Big)\|\nabla^3d\|_{L^2}^2
+C\|\nabla d\|_{H^1}^2+C\|\nabla u\|_{L^2}^2\|\nabla^2d\|_{L^2}^2
+C\|\nabla u\|_{L^2}^4\|\nabla^2d\|_{L^2}^2,
\end{align*}
which leads to
\begin{align}\label{t3.34}
\|\nabla^3d\|_{L^2}^2\le C\|\nabla d_t\|_{L^2}^2+C\|\nabla d\|_{H^1}^2+C\|\nabla u\|_{L^2}^2\|\nabla^2d\|_{L^2}^2
+C\|\nabla u\|_{L^2}^4\|\nabla^2d\|_{L^2}^2,
\end{align}
provided that $E_0\le \varepsilon_1$ is suitably small. Substituting \eqref{t3.34} into \eqref{t3.33}, one has
\begin{align*}
&\frac{d}{dt}(\eta(t)\|\Delta d\|_{L^2}^2)+\frac12\eta(t)\|\nabla d_t\|_{L^2}^2\nonumber\\
&\le \eta'(t)\|\Delta d\|_{L^2}^2
+C\eta(t)\|\nabla u\|_{L^2}^2\|\nabla^2d\|_{L^2}^2+C\eta(t)\|\nabla u\|_{L^2}^4\|\nabla^2d\|_{L^2}^2\notag \\
& \quad +C\eta(t)\|\nabla u\|_{L^3}^3+C\eta(t)\|\nabla d\|_{H^1}^2+C\eta(t)\|\nabla d\|_{H^1}^6.
\end{align*}
This together with \eqref{2.35} and \eqref{t3.34} gives \eqref{t3.20} after choosing $E_0\le \varepsilon_2\le \varepsilon_1$ and $\delta$ sufficiently small.
\end{proof}

\begin{lemma}\label{zl33}
Let $(\rho, u, d)$ be a strong solution
of \eqref{a1}--\eqref{a6} satisfying \eqref{3.6} and $\eta(t)$ be as in Lemma \ref{l33}, then it holds that
\begin{align}\label{t3.36}
&\frac{d}{dt}\Big(\frac{\eta(t)}{2}\|\sqrt{\rho}\dot{u}\|_{L^2}^2+\frac{\eta(t)}{2}\|\nabla d_t\|_{L^2}^2\Big)
+(2\mu+\lambda)\eta(t)\|{\rm div}\,\dot{u}\|_{L^2}^2+\mu\eta(t)\|{\rm curl}\,\dot{u}\|_{L^2}^2+\eta(t)\|d_{tt}\|_{L^2}^2\nonumber\\
&\le -C\frac{d}{dt}\int_{\partial\Omega}\eta(t)(u\cdot\nabla n\cdot u)FdS+C|\eta'(t)|\big(\|\nabla u\|_{L^2}^4+\|\nabla^2d\|_{L^2}^4+\|\nabla^2d\|_{L^2}^6+E_0\big)\nonumber\\
&\quad+C|\eta'(t)|\big(\|\sqrt{\rho}\dot{u}\|_{L^2}^2+\|\nabla u\|_{L^2}^2+\|\nabla^3d\|_{L^2}^2
+CE_0^2\|\nabla^2d\|_{L^2}^2+\|\nabla d_t\|_{L^2}^2\big)\nonumber\\
&\quad+C\eta(t)\big(\|\sqrt{\rho}\dot{u}\|_{L^2}^2\|\nabla u\|_{L^2}^2
+\|\nabla u\|_{L^2}^4\|\nabla^3d\|_{L^2}^2+\|\nabla u\|_{L^2}^2\|\nabla^3d\|_{L^2}^2
\big)\nonumber\\
&\quad+C\eta(t)\big(\|\nabla u\|_{L^2}^4+E_0^2\|\nabla u\|_{L^2}^2
+\|\nabla u\|_{L^4}^4+\|\nabla u\|_{L^2}^6+\|\nabla^2d\|_{L^2}^4\|\nabla u\|_{L^2}^2\big)\nonumber\\
&\quad+C\delta\eta(t)\big(\|\nabla^2d\|_{L^2}\|\nabla^3d\|_{L^2}^3+\|\nabla^2d\|_{L^2}^2\|\nabla^3d\|_{L^2}^2
+\|\nabla^2d\|_{L^2}^4
+\|\nabla^2d\|_{L^2}^3\|\nabla^3d\|_{L^2}\big)\nonumber\\
&\quad+C\eta(t)\big(\|\nabla u\|_{L^2}^2\|\nabla d\|_{H^1}^4
+\|\nabla u\|_{L^2}^2\|\nabla d\|_{H^1}^2\|\nabla^3d\|_{L^2}^2
+\|\nabla d\|_{H^1}^2\|\nabla d_t\|_{L^2}^2\big)\nonumber\\
&\quad+C\eta(t)\big(\|\nabla u\|_{L^2}^4\|\nabla d_t\|_{L^2}^2+\|\nabla u\|_{L^2}^2\|\nabla d_t\|_{L^2}^2\big)
+C\eta(t)\|\nabla d\|_{H^1}^4\big(\|d_t\|_{L^2}^2+\|\nabla d_t\|_{L^2}^2\big),
\end{align}
provided that $E_0\leq\varepsilon_3$.
\end{lemma}
\begin{proof}[Proof]
1. By \eqref{2.7} and \eqref{t3.8}, we rewrite $\eqref{a2}$ as
\begin{align}\label{3.41}
\rho\dot{u}=\nabla F-\mu\curl{\rm curl}\,u-\Delta d\cdot\nabla d.
\end{align}
Applying $\eta(t)\dot{u}^j[\partial/\partial t+\divv(u\cdot)]$ to the $j$th-component of $\eqref{3.41}$,
 and then integrating the resulting equality over $\Omega$, we get that
\begin{align}\label{3.44}
&\frac12\left(\frac{d}{dt}\int\eta(t)\rho|\dot{u}|^2dx
-\eta'(t)\int\rho|\dot{u}|^2dx\right)\nonumber\\
&=\int\eta(t)\big(\dot{u}\cdot\nabla F_t+\dot{u}^j{\rm div}\,(u\partial_jF)\big)dx\nonumber\\
&\quad-\mu\int\eta(t)\big(\dot{u}\cdot\curl{\rm curl}\,u_t+\dot{u}^j{\rm div}\,((\curl{\rm curl}\,u)^ju)\big)dx\nonumber\\
&\quad+\int\eta(t)\big(M(d):\nabla u_t+\divv M(d)\cdot(u\cdot\nabla\dot{u})\big)dx
\triangleq\sum_{i=1}^3J_i.
\end{align}
We denote by $h\triangleq u\cdot(\nabla n+(\nabla n)^{tr})$ and $u^\bot\triangleq-u\times n$, then it deduces from Lemma \ref{l23} that
\begin{align}\label{z3.22}
&-\int_{\partial\Omega}\eta(t)F_t(u\cdot\nabla n\cdot u)dS\nonumber\\
&=-\frac{d}{dt}\int_{\partial\Omega}\eta(t)(u\cdot\nabla n\cdot u)FdS
+\int_{\partial\Omega}\eta(t)F_1h\cdot\dot{u}dS-\int_{\partial\Omega}\eta(t)
Fh\cdot(u\cdot\nabla u)dS\nonumber\\
&\quad+\eta'(t)\int_{\partial\Omega}(u\cdot\nabla n\cdot u)FdS\nonumber\\
&=-\frac{d}{dt}\int_{\partial\Omega}\eta(t)(u\cdot\nabla n\cdot u)FdS
+\int_{\partial\Omega}\eta(t)Fh\cdot\dot{u}dS+\eta'(t)\int_{\partial\Omega}(u\cdot\nabla n\cdot u)FdS\nonumber\\
&\quad-\int_{\partial\Omega}\eta(t)Fh^i(\nabla u^i\times u^\bot)\cdot ndS\nonumber\\
&=-\frac{d}{dt}\int_{\partial\Omega}\eta(t)(u\cdot\nabla n\cdot u)FdS
+\int_{\partial\Omega}\eta(t)Fh\cdot\dot{u}dS
+\eta'(t)\int_{\partial\Omega}(u\cdot\nabla n\cdot u)FdS\nonumber\\
&\quad-\int\eta(t)\nabla u^i\times u^\bot\cdot\nabla(Fh^i)dx+\int\eta(t)Fh^i\nabla\times u^\bot\cdot\nabla u^idx\nonumber\\
&\le -\frac{d}{dt}\int_{\partial\Omega}\eta(t)(u\cdot\nabla n\cdot u)FdS
+C\eta(t)\|\nabla F\|_{L^2}\|u\|_{L^3}\|\dot{u}\|_{L^6}\nonumber\\
&\quad+C\eta(t)\big(\|F\|_{L^3}\|u\|_{L^6}\|\nabla\dot{u}\|_{L^2}
+\|F\|_{L^3}\|u\|_{L^6}\|\dot{u}\|_{L^2}
+\|F\|_{L^3}\|\nabla u\|_{L^2}\|\dot{u}\|_{L^6}\big)\nonumber\\
&\quad+C\eta(t)\big(\|\nabla u\|_{L^2}\|u\|_{L^6}^2\|\nabla F\|_{L^6}+\|\nabla u\|_{L^4}^2\|u\|_{L^6}\|F\|_{L^3}\big)\nonumber\\
&\quad+|\eta'(t)|\big(\|\nabla u\|_{L^2}\|u\|_{L^6}\|F\|_{L^3}+\|u\|_{L^4}^2\|F\|_{L^2}
+\|\nabla F\|_{L^2}\|u\|_{L^4}^2\big)\nonumber\\
&\le -\frac{d}{dt}\int_{\partial\Omega}\eta(t)(u\cdot\nabla n\cdot u)FdS
+C|\eta'(t)|\|\nabla u\|_{L^2}^2\big(\|\rho\dot{u}\|_{L^2}+\|\nabla d\|_{L^6}\|\nabla^2d\|_{L^3}+\|\nabla u\|_{L^2}+\|P-\bar{P}\|_{L^2}\big)\nonumber\\
&\quad+C\eta(t)\big(\|\rho\dot{u}\|_{L^2}+\|\nabla d\|_{L^6}\|\nabla^2d\|_{L^3}
+\|\nabla u\|_{L^2}+\|P-\bar{P}\|_{L^2}\big)\|\nabla u\|_{L^2}\big(\|\nabla\dot{u}\|_{L^2}
+\|\nabla u\|_{L^2}^2+\|\nabla u\|_{L^4}^2\big)\nonumber\\
&\quad+C\eta(t)\|\nabla u\|_{L^2}^3\|\nabla F\|_{L^6}\nonumber\\
&\le  -\frac{d}{dt}\int_{\partial\Omega}\eta(t)(u\cdot\nabla n\cdot u)FdS
+C|\eta'(t)|\big(\|\sqrt{\rho}\dot{u}\|_{L^2}^2+\|\nabla u\|_{L^2}^2+\|\nabla^3d\|_{L^2}^2+CE_0^2\|\nabla^2d\|_{L^2}^2\big)\nonumber\\
&\quad+C|\eta'(t)|(\|\nabla u\|_{L^2}^4+\|\nabla^2d\|_{L^2}^4+\|\nabla^2d\|_{L^2}^6+E_0)+\delta\eta(t)
\big(\|\nabla\dot{u}\|_{L^2}^2+\|\nabla F\|_{L^6}^2\big)\nonumber\\
&\quad+C\eta(t)\big(\|\sqrt{\rho}\dot{u}\|_{L^2}^2\|\nabla u\|_{L^2}^2
+\|\nabla u\|_{L^2}^4\|\nabla^3d\|_{L^2}^2+\|\nabla u\|_{L^2}^2\|\nabla^3d\|_{L^2}^2+\|\nabla^2d\|_{L^2}^4\|\nabla u\|_{L^2}^2
\big)\nonumber\\
&\quad+C\eta(t)\big(\|\nabla u\|_{L^2}^4+E_0^2\|\nabla u\|_{L^2}^2
+\|\nabla u\|_{L^4}^4+\|\nabla u\|_{L^2}^6\big),
\end{align}
due to
\begin{align*}
&{\rm div}(\nabla u^i\times u^\bot)=u^\bot\cdot\curl\nabla u^i-\nabla u^i\cdot\curl u^\bot=-\nabla u^i\cdot\curl u^\bot,\\
&\|\dot{u}\|_{L^6}\le C\big(\|\nabla\dot{u}\|_{L^2}+\|\nabla u\|_{L^2}^2\big)\ \ (\text{see}\ \eqref{3.2}),
\end{align*}
and
\begin{align}
\|\nabla d\|_{L^6}\|\nabla^2d\|_{L^3}&\le C(\|\nabla d\|_{L^2}+\|\nabla^2d\|_{L^2})(\|\nabla^2d\|_{L^2}^\frac12
\|\nabla^3d\|_{L^2}^\frac12+\|\nabla^2d\|_{L^2})\nonumber\\
&\le CE_0^\frac12\|\nabla^2d\|_{L^2}^\frac12\|\nabla^3d\|_{L^2}^\frac12
+C\|\nabla^2d\|_{L^2}^2+C\|\nabla^2d\|_{L^2}^\frac32\|\nabla^3d\|_{L^2}^\frac12
+CE_0.
\end{align}
Thus, it follows from integration by parts, \eqref{a6}, \eqref{2.5}, \eqref{3.6}, H{\"o}lder's inequality, \eqref{2.3}, Lemma \ref{l28}, and \eqref{z3.22} that
\begin{align}\label{3.45}
J_1&=\int_{\partial\Omega}\eta(t)F_t\dot{u}\cdot ndS-\int\eta(t)F_t{\rm div}\,\dot{u}dx
-\int\eta(t)u\cdot\nabla\dot{u}\cdot\nabla Fdx\nonumber\\
&=-\int_{\partial\Omega}\eta(t)F_t(u\cdot\nabla n\cdot u)dS-(2\mu+\lambda)\int\eta(t)({\rm div}\,\dot{u})^2dx
+(2\mu+\lambda)\int\eta(t){\rm div}\,\dot{u}\nabla u:\nabla udx\nonumber\\
&\quad-\gamma\int\eta(t) P{\rm div}\,\dot{u}{\rm div}\,udx
+\int\eta(t){\rm div}\,\dot{u}u\cdot\nabla Fdx-\int\eta(t)u\cdot\nabla\dot{u}\cdot\nabla Fdx\nonumber\\
&\le -\int_{\partial\Omega}\eta(t)F_t(u\cdot\nabla n\cdot u)dS
-(2\mu+\lambda)\int\eta(t)({\rm div}\,\dot{u})^2dx+\delta\eta(t)\|\nabla\dot{u}\|_{L^2}^2\nonumber\\
&\quad+C(\delta)\eta(t)\big(\|\nabla u\|_{L^2}^2\|\nabla F\|_{L^3}^2
+\|\nabla u\|_{L^4}^4+\|\nabla u\|_{L^2}^2\big)\nonumber\\
&\le -\int_{\partial\Omega}\eta(t)F_t(u\cdot\nabla n\cdot u)dS
-(2\mu+\lambda)\int\eta(t)({\rm div}\,\dot{u})^2dx+\delta\eta(t)\|\nabla\dot{u}\|_{L^2}^2\nonumber\\
&\quad+C\eta(t)\big(\|\nabla u\|_{L^2}^2\|\nabla F\|_{L^2}\|\nabla F\|_{L^6}+\|\nabla u\|_{L^4}^4+\|\nabla u\|_{L^2}^2
\big)\nonumber\\
&\le -\int_{\partial\Omega}\eta(t)F_t(u\cdot\nabla n\cdot u)dS
-(2\mu+\lambda)\int\eta(t)({\rm div}\,\dot{u})^2dx+\delta\eta(t)\big(\|\nabla\dot{u}\|_{L^2}^2+\|\nabla F\|_{L^6}^2\big)\nonumber\\
&\quad+C\eta(t)\big(\|\nabla u\|_{L^2}^4\|\nabla F\|_{L^2}^2+\|\nabla u\|_{L^4}^4+\|\nabla u\|_{L^2}^2
\big)\nonumber\\
&\le -\frac{d}{dt}\int_{\partial\Omega}(u\cdot\nabla n\cdot u)FdS-(2\mu+\lambda)\int\eta(t)({\rm div}\,\dot{u})^2dx+\delta\eta(t)\big(\|\nabla\dot{u}\|_{L^2}^2+\|\nabla F\|_{L^6}^2\big)\nonumber\\
&\quad+C|\eta'(t)|\big(\|\sqrt{\rho}\dot{u}\|_{L^2}^2+\|\nabla u\|_{L^2}^2+\|\nabla^3d\|_{L^2}^2
+CE_0^2\|\nabla^2d\|_{L^2}^2\big)\nonumber\\
&\quad+C|\eta'(t)|\big(\|\nabla u\|_{L^2}^4+\|\nabla^2d\|_{L^2}^4+\|\nabla^2d\|_{L^2}^6+E_0\big)\nonumber\\
&\quad+C\eta(t)\big(\|\sqrt{\rho}\dot{u}\|_{L^2}^2\|\nabla u\|_{L^2}^2
+\|\nabla u\|_{L^2}^4\|\nabla^3d\|_{L^2}^2+\|\nabla u\|_{L^2}^2\|\nabla^3d\|_{L^2}^2
\big)\nonumber\\
&\quad+C\eta(t)\big(\|\nabla u\|_{L^2}^4+E_0^2\|\nabla u\|_{L^2}^2
+\|\nabla u\|_{L^4}^4+\|\nabla u\|_{L^2}^6+\|\nabla^2d\|_{L^2}^4\|\nabla u\|_{L^2}^2\big)\nonumber\\
&\le -\frac{d}{dt}\int_{\partial\Omega}(u\cdot\nabla n\cdot u)FdS-(2\mu+\lambda)\int\eta(t)({\rm div}\,\dot{u})^2dx+\delta\eta(t)\|\nabla\dot{u}\|_{L^2}^2\nonumber\\
&\quad+C|\eta'(t)|\big(\|\sqrt{\rho}\dot{u}\|_{L^2}^2+\|\nabla u\|_{L^2}^2+\|\nabla^3d\|_{L^2}^2
+CE_0^2\|\nabla^2d\|_{L^2}^2\big)\nonumber\\
&\quad+C|\eta'(t)|\big(\|\nabla u\|_{L^2}^4+\|\nabla^2d\|_{L^2}^4+\|\nabla^2d\|_{L^2}^6+E_0\big)\nonumber\\
&\quad+C\eta(t)\big(\|\sqrt{\rho}\dot{u}\|_{L^2}^2\|\nabla u\|_{L^2}^2
+\|\nabla u\|_{L^2}^4\|\nabla^3d\|_{L^2}^2+\|\nabla u\|_{L^2}^2\|\nabla^3d\|_{L^2}^2
\big)\nonumber\\
&\quad+C\eta(t)\big(\|\nabla u\|_{L^2}^4+E_0^2\|\nabla u\|_{L^2}^2
+\|\nabla u\|_{L^4}^4+\|\nabla u\|_{L^2}^6+\|\nabla^2d\|_{L^2}^4\|\nabla u\|_{L^2}^2\big)\nonumber\\
&\quad+C\delta\eta(t)\big(\|\nabla^2d\|_{L^2}\|\nabla^3d\|_{L^2}^3+\|\nabla^2d\|_{L^2}^2\|\nabla^3d\|_{L^2}^2
+\|\nabla^2d\|_{L^2}^4
+\|\nabla^2d\|_{L^2}^3\|\nabla^3d\|_{L^2}\big),
\end{align}
where we have used
\begin{align*}
F_t&=(2\mu+\lambda){\rm div}\,u_t-(P-\bar{P})_t\nonumber\\
&=(2\mu+\lambda){\rm div}\,\dot{u}-(2\mu+\lambda){\rm div}\,(u\cdot\nabla u)+u\cdot\nabla P+\gamma P\divv u
\nonumber\\
&=(2\mu+\lambda){\rm div}\,\dot{u}-(2\mu+\lambda)u\cdot\nabla{\rm div}\,u
-(2\mu+\lambda)\nabla u:\nabla u+u\cdot\nabla P
+\gamma P\divv u\nonumber\\
&=(2\mu+\lambda){\rm div}\,\dot{u}
-(2\mu+\lambda)\nabla u:\nabla u+\gamma P{\rm div}\,u-u\cdot\nabla F,
\end{align*}
and
\begin{align*}
\|\nabla F\|_{L^6}&\le C\big(\|\rho\dot{u}\|_{L^6}+\||\nabla d||\nabla^2d|\|_{L^6}\big)\nonumber\\
&\le C(\hat{\rho})\big(\|\dot{u}\|_{L^6}+\|\nabla d\|_{L^\infty}\|\nabla^2d\|_{L^6}\big)\nonumber\\
&\le C\Big(\|\nabla^2d\|_{L^2}^\frac12\|\nabla^3d\|_{L^2}^\frac12
+\|\nabla^2d\|_{L^2}\Big)\big(\|\nabla^3d\|_{L^2}+\|\nabla^2d\|_{L^2}\big)
+C(\hat{\rho})\big(\|\nabla\dot{u}\|_{L^2}+\|\nabla u\|_{L^2}^2\big)\nonumber\\
&\le C\Big(\|\nabla^2d\|_{L^2}^\frac12\|\nabla^3d\|_{L^2}^\frac32
+\|\nabla^2d\|_{L^2}\|\nabla^3d\|_{L^2}+\|\nabla^2d\|_{L^2}^2\Big)\nonumber\\
&\quad+C(\hat{\rho})\Big(\|\nabla^2d\|_{L^2}^\frac32\|\nabla^3d\|_{L^2}^\frac12
+\|\nabla\dot{u}\|_{L^2}+\|\nabla u\|_{L^2}^2\Big).
\end{align*}

2. By a direct calculation, one obtains that
\begin{align}
J_2&=-\mu\int\eta(t)\dot{u}\cdot(\curl{\rm curl}\,u_t)dx-\mu\int\eta(t)\dot{u}\cdot(\curl{\rm curl}\,u){\rm div}\,udx\nonumber\\
&\quad-\mu\int\eta(t)u^i\dot{u}\cdot\curl(\partial_i{\rm curl}\,u)dx\nonumber\\
&=-\mu\int\eta(t)|{\rm curl}\,\dot{u}|^2dx+\mu\int\eta(t){\rm curl}\,\dot{u}\cdot{\rm curl}\,(u\cdot\nabla u)dx\nonumber\\
&\quad+\mu\int\eta(t)({\rm curl}\,u\times\dot{u})\cdot\nabla{\rm div}\,udx
-\mu\int\eta(t){\rm div}\,u({\rm curl}\,u\cdot{\rm curl}\,\dot{u})dx\nonumber\\
&\quad-\mu\int\eta(t)u^i{\rm div}\,(\partial_i{\rm curl}\,u\times\dot{u})dx
-\mu\int\eta(t)u^i\partial_i{\rm curl}\,u\cdot{\rm curl}\,\dot{u}dx\nonumber\\
&=-\mu\int\eta(t)|{\rm curl}\,\dot{u}|^2dx+\mu\int\eta(t){\rm curl}\,\dot{u}\partial_iu\times\nabla u^idx\nonumber\\
&\quad+\mu\int\eta(t)({\rm curl}\,u\times\dot{u})\cdot\nabla{\rm div}\,udx
-\mu\int\eta(t){\rm div}\,u({\rm curl}\,u\cdot{\rm curl}\,\dot{u})dx\nonumber\\
&\quad-\mu\int\eta(t)u^i{\rm div}\,(\partial_i{\rm curl}\,u\times\dot{u})dx\nonumber\\
&=-\mu\int\eta(t)|{\rm curl}\,\dot{u}|^2dx+\mu\int\eta(t){\rm curl}\,\dot{u}\partial_iu\times\nabla u^idx\nonumber\\
&\quad+\mu\int\eta(t)({\rm curl}\,u\times\dot{u})\cdot\nabla{\rm div}\,udx
-\mu\int\eta(t){\rm div}\,u({\rm curl}\,u\cdot{\rm curl}\,\dot{u})dx\nonumber\\
&\quad-\mu\int\eta(t) u\cdot\nabla{\rm div}\,({\rm curl}\,u\times\dot{u})dx
+\mu\int\eta(t) u^i{\rm div}\,({\rm curl}\,u\times\partial_i\dot{u})dx\nonumber\\
&=-\mu\int\eta(t)|{\rm curl}\,\dot{u}|^2dx+\mu\int\eta(t){\rm curl}\,\dot{u}\nabla_iu\times\nabla u^idx\nonumber\\
&\quad-\mu\int\eta(t){\rm div}\,u({\rm curl}\,u\cdot{\rm curl}\,\dot{u})dx
-\mu\int\eta(t)\nabla u^i\cdot({\rm curl}\,u\times\partial_i\dot{u})dx\nonumber\\
&\le \delta\eta(t)\|\nabla\dot{u}\|_{L^2}^2+C\eta(t)\|\nabla u\|_{L^4}^4-\mu\eta(t)\|{\rm curl}\,\dot{u}\|_{L^2}^2,
\end{align}
due to
\begin{align*}
\curl(\dot{u}{\rm div}\,u)&={\rm div}\,u{\rm curl}\,\dot{u}+\nabla{\rm div}\,u\times\dot{u},\\
{\rm div}\,(\partial_i{\rm curl}\,u\times\dot{u})&=\dot{u}\cdot\curl(\partial_i{\rm curl}\,u)-\partial_i{\rm curl}\,u\cdot\curl\dot{u},\\
\int{\rm curl}\,u\cdot(\nabla{\rm div}\,u\times\dot{u})dx&=-\int({\rm curl}\,u\times\dot{u})\cdot\nabla{\rm div}\,udx,\\
\int{\rm curl}\,\dot{u}\cdot{\rm curl}\,(u\cdot\nabla u)dx&=\int{\rm curl}\,\dot{u}\cdot{\rm curl}\,(u^i\partial_iu)dx \\
& =\int{\rm curl}\,\dot{u}\big(u^i{\rm curl}\,\partial_iu+\partial_iu\cdot\nabla u^i\big)dx\nonumber\\
&=\int u^i\partial_i{\rm curl}\,u\cdot{\rm curl}\,\dot{u}dx+\int{\rm curl}\,\dot{u}\partial_iu\times\nabla u^idx,
\end{align*}
and
\begin{align*}
\int u\cdot\nabla{\rm div}\,({\rm curl}\,u\times\dot{u})dx
&=\int u^i\partial_i{\rm div}\,({\rm curl}\,u\times\dot{u})dx\nonumber\\
&=\int u^i\partial_i(\dot{u}\cdot\curl{\rm curl}\,u-{\rm curl}\,\dot{u}\cdot{\rm curl}\,u)dx\nonumber\\
&=\int u^i(\dot{u}\cdot\partial_i\curl{\rm curl}\,u-{\rm curl}\,\dot{u}\cdot\partial_i{\rm curl}\,u)dx\nonumber\\
&\quad+\int u^i(\partial_i\dot{u}\cdot\curl{\rm curl}\,udx-\partial_i{\rm curl}\,\dot{u}\cdot{\rm curl}\,u)dx\nonumber\\
&=\int u^i{\rm div}\,(\partial_i{\rm curl}\,u\times\dot{u})dx+\int u^i{\rm div}\,({\rm curl}\,u\times\partial_i\dot{u})dx.
\end{align*}
By H\"older's inequality, \eqref{3.2}, Sobolev's inequality, \eqref{2.3}, and \eqref{3.8}, we derive from Lemmas \ref{l24} and \ref{l25} that
\begin{align}
J_3&\le C\eta(t)\|\nabla\dot{u}\|_{L^2}\big(\|\nabla d\|_{L^6}\|\nabla d_t\|_{L^3}
+\|\nabla d\|_{L^6}\|\nabla^2d\|_{L^6}\|u\|_{L^6}\big)\nonumber\\
&\le C\eta(t)\|\nabla\dot{u}\|_{L^2}\|\nabla d\|_{H^1}\big(\|\nabla d_t\|_{L^2}^\frac12\|\nabla^2d_t\|_{L^2}^\frac12+\|\nabla d_t\|_{L^2}\big)\nonumber\\
&\quad+C\eta(t)\|\nabla\dot{u}\|_{L^2}\|\nabla u\|_{L^2}\big(\|\nabla d\|_{L^2}+\|\nabla^2d\|_{L^2}\big)
\big(\|\nabla^2d\|_{L^2}+\|\nabla^3d\|_{L^2}\big)\nonumber\\
&\le \delta\eta(t)\big(\|\nabla\dot{u}\|_{L^2}^2+\|\nabla^2d_t\|_{L^2}^2\big)
+C\eta(t)\|\nabla d\|_{H^1}^2\|\nabla d_t\|_{L^2}^2\nonumber\\
&\quad+C\eta(t)\big(\|\nabla u\|_{L^2}^2\|\nabla d\|_{H^1}^4
+\|\nabla u\|_{L^2}^2\|\nabla d\|_{H^1}^2\|\nabla^3d\|_{L^2}^2\big).
\end{align}
Putting above estimates on $J_1$, $J_2$, and $J_3$ into \eqref{3.44}, we obtain after choosing $\delta$ suitably small that
\begin{align}\label{3.59}
&\frac{d}{dt}\Big(\frac{\eta(t)}{2}\|\sqrt{\rho}\dot{u}\|_{L^2}^2\Big)
+(2\mu+\lambda)\eta(t)\|{\rm div}\,\dot{u}\|_{L^2}^2+\mu\eta(t)\|{\rm curl}\,\dot{u}\|_{L^2}^2\nonumber\\
&\le -\frac{d}{dt}\int_{\partial\Omega}(u\cdot\nabla n\cdot u)FdS+C|\eta'(t)|\big(\|\nabla u\|_{L^2}^4+\|\nabla^2d\|_{L^2}^4+\|\nabla^2d\|_{L^2}^6+E_0\big)\nonumber\\
&\quad+C|\eta'(t)|\big(\|\sqrt{\rho}\dot{u}\|_{L^2}^2+\|\nabla u\|_{L^2}^2+\|\nabla^3d\|_{L^2}^2
+CE_0^2\|\nabla^2d\|_{L^2}^2\big)\nonumber\\
&\quad+C\eta(t)\big(\|\sqrt{\rho}\dot{u}\|_{L^2}^2\|\nabla u\|_{L^2}^2
+\|\nabla u\|_{L^2}^4\|\nabla^3d\|_{L^2}^2+\|\nabla u\|_{L^2}^2\|\nabla^3d\|_{L^2}^2\big)\nonumber\\
&\quad+C\eta(t)\big(\|\nabla u\|_{L^2}^4+E_0^2\|\nabla u\|_{L^2}^2
+\|\nabla u\|_{L^4}^4+\|\nabla u\|_{L^2}^6+\|\nabla^2d\|_{L^2}^4\|\nabla u\|_{L^2}^2\big)\nonumber\\
&\quad+C\delta\eta(t)\big(\|\nabla^2d\|_{L^2}\|\nabla^3d\|_{L^2}^3+\|\nabla^2d\|_{L^2}^2\|\nabla^3d\|_{L^2}^2
+\|\nabla^2d\|_{L^2}^4
+\|\nabla^2d\|_{L^2}^3\|\nabla^3d\|_{L^2}\big)\nonumber\\
&\quad+C\eta(t)\big(\|\nabla u\|_{L^2}^2\|\nabla d\|_{H^1}^4
+\|\nabla u\|_{L^2}^2\|\nabla d\|_{H^1}^2\|\nabla^3d\|_{L^2}^2
+\|\nabla d\|_{H^1}^2\|\nabla d_t\|_{L^2}^2\big).
\end{align}

3. Differentiating \eqref{a3} with respect to $t$ and multiplying the resulting equations by $d_{tt}$,
we obtain from integration by parts, $\frac{\partial d_t}{\partial n}|_{\partial\Omega}=0$, Sobolev's inequality, and H\"older's inequality that \begin{align}\label{t3.45}
&\frac12\frac{d}{dt}\int|\nabla d_t|^2dx+\int|d_{tt}|^2dx\nonumber\\
&=\int\langle\partial_t\big(|\nabla d|^2d-u\cdot\nabla d\big), d_{tt}\rangle dx\nonumber\\
&\le C\int|d_{tt}||u_t||\nabla d|dx+C\int|d_{tt}||u||\nabla d_t|dx+C\int|d_{tt}||d_t||\nabla d|^2dx
+C\int|d_{tt}||\nabla d_t||\nabla d|dx\nonumber\\
&\le C\int|d_{tt}||\dot{u}||\nabla d|dx+C\int|d_{tt}||u||\nabla d_t|dx+C\int|d_{tt}||d_t||\nabla d|^2dx+C\int|d_{tt}||\nabla d_t||\nabla d|dx\nonumber\\
&\quad+C\int|d_{tt}||u||\nabla u||\nabla d|dx\triangleq\sum_{i=1}^5U_i.
\end{align}
By a direct computation, one has
\begin{align*}
U_1&\le \delta\|d_{tt}\|_{L^2}^2+C\|\dot{u}\|_{L^6}^2\|\nabla d\|_{L^3}^2\nonumber\\
&\le \delta\|d_{tt}\|_{L^2}^2+C\big(\|\nabla\dot{u}\|_{L^2}^2+\|\nabla u\|_{L^2}^4\big)
\big(\|\nabla d\|_{L^2}\|\nabla^2d\|_{L^2}+\|\nabla d\|_{L^2}^2\big),\\
&\le \delta\|d_{tt}\|_{L^2}^2+CE_0^\frac12\|\nabla\dot{u}\|_{L^2}^2+C\|\nabla u\|_{L^2}^4,\\
U_2&\le \delta\|d_{tt}\|_{L^2}^2+C\|u\|_{L^6}^2\|\nabla d_t\|_{L^3}^2\nonumber\\
&\le \delta\|d_{tt}\|_{L^2}^2+C\|\nabla u\|_{L^2}^2\big(\|\nabla d_t\|_{L^2}\|\nabla^2d_t\|_{L^2}+\|\nabla d_t\|_{L^2}^2\big)\nonumber\\
&\le \delta\big(\|d_{tt}\|_{L^2}^2+\|\nabla^2d_t\|_{L^2}^2\big)
+C\|\nabla u\|_{L^2}^4\|\nabla d_t\|_{L^2}^2+C\|\nabla u\|_{L^2}^2\|\nabla d_t\|_{L^2}^2,\\
U_3&\le \delta\|d_{tt}\|_{L^2}^2+C\|d_t\|_{L^6}^2\|\nabla d\|_{L^6}^4\nonumber\\
&\le \delta\|d_{tt}\|_{L^2}^2+C\|\nabla d\|_{H^1}^4\big(\|d_t\|_{L^2}^2+\|\nabla d_t\|_{L^2}^2\big),\\
U_4&\le \delta\|d_{tt}\|_{L^2}^2+C\|\nabla d_t\|_{L^3}^2\|\nabla d\|_{L^6}^2\nonumber\\
&\le  \delta\|d_{tt}\|_{L^2}^2+C\big(\|\nabla d_t\|_{L^2}\|\nabla^2d_t\|_{L^2}+\|\nabla d_t\|_{L^2}^2\big)\|\nabla d\|_{H^1}^2\nonumber\\
&\le  \delta\big(\|d_{tt}\|_{L^2}^2+\|\nabla^2d_t\|_{L^2}^2\big)
+C\|\nabla d\|_{H^1}^4\|\nabla d_t\|_{L^2}^2+C\|\nabla d\|_{H^1}^2\|\nabla d_t\|_{L^2}^2,\\
U_5&\le \delta\|d_{tt}\|_{L^2}^2+C\|u\|_{L^6}^2\|\nabla u\|_{L^4}^2\|\nabla d\|_{L^{12}}^2\nonumber\\
&\le  \delta\|d_{tt}\|_{L^2}^2+C\|\nabla u\|_{L^4}^4
+C\|\nabla u\|_{L^2}^4\big(\|\nabla d\|_{H^1}^3\|\nabla^3d\|_{L^2}+\|\nabla d\|_{H^1}^4\big)\nonumber\\
&\le \delta\|d_{tt}\|_{L^2}^2+C\|\nabla u\|_{L^4}^4+C\|\nabla u\|_{L^2}^4+C\|\nabla u\|_{L^2}^4\|\nabla^3d\|_{L^2}^2,
\end{align*}
where we have used \eqref{2.5}, \eqref{3.8}, \eqref{t3.4}, and
\begin{align*}
\|\nabla d\|_{L^{12}}^2&\le C\||\nabla d||\nabla^2d|\|_{L^2}\le C\|\nabla d\|_{L^6}\|\nabla^2d\|_{L^3}\nonumber\\
&\le C\|\nabla d\|_{H^1}
\Big(\|\nabla^2d\|_{L^2}^\frac12\|\nabla^3d\|_{L^2}^\frac12
+\|\nabla^2d\|_{L^2}\Big)\nonumber\\
&\le C\|\nabla d\|_{H^1}^\frac32\|\nabla^3d\|_{L^2}^\frac12+\|\nabla d\|_{H^1}^2,
\end{align*}
due to Lemma \ref{l22}.
Thus, substituting the above estimates on $U_i\ (i=1, 2, \ldots, 5)$ into \eqref{t3.45} shows that
\begin{align}\label{t3.46}
&\frac{d}{dt}\Big(\frac{\eta(t)}{2}\|\nabla d_t\|_{L^2}^2\Big)
+\eta(t)\|d_{tt}\|_{L^2}^2-\frac12\eta'(t)\|\nabla d_t\|_{L^2}^2\nonumber\\
&\le C\delta\eta(t)(\|d_{tt}\|_{L^2}^2+\|\nabla^2d_t\|_{L^2}^2)+CE_0^\frac12\eta(t)\|\nabla\dot{u}\|_{L^2}^2
+C\eta(t)\|\nabla u\|_{L^4}^4+C\eta(t)\|\nabla u\|_{L^2}^4\nonumber\\
&\quad+C\eta(t)\|\nabla u\|_{L^2}^4\|\nabla d_t\|_{L^2}^2+C\eta(t)\|\nabla u\|_{L^2}^2\|\nabla d_t\|_{L^2}^2
+C\eta(t)\|\nabla d\|_{H^1}^4(\|d_t\|_{L^2}^2+\|\nabla d_t\|_{L^2}^2)\nonumber\\
&\quad+C\eta(t)\|\nabla d\|_{H^1}^2\|\nabla d_t\|_{L^2}^2+C\eta(t)\|\nabla u\|_{L^2}^4\|\nabla^3d\|_{L^2}^2.
\end{align}

4. It remains to estimate $\|\nabla^2d_t\|_{L^2}$. In fact, by applying the standard $L^2$-estimate of \eqref{a3}, we obtain from \eqref{2.3} and \eqref{3.8} that
\begin{align}
\|\nabla^2d_t\|_{L^2}^2&\le C(\|\nabla d_t\|_{L^2}^2+\|d_{tt}\|_{L^2}^2+\|\partial_t(u\cdot\nabla d)\|_{L^2}^2
+\|\partial_t(|\nabla d|^2d)\|_{L^2}^2)\nonumber\\
&\le C\|\nabla d_t\|_{L^2}^2+C\|d_{tt}\|_{L^2}^2+C\|\dot{u}\|_{L^6}^2\|\nabla d\|_{L^3}^2+C\|d_t\|_{L^6}^2\|\nabla d\|_{L^6}^4\nonumber\\
&\quad+C\|\nabla d_t\|_{L^3}^2\|\nabla d\|_{L^6}^2+C\|u\|_{L^6}^2\|\nabla u\|_{L^4}^2\|\nabla d\|_{L^{12}}^2\nonumber\\
&\le \frac12\|\nabla^2d_t\|_{L^2}^2+C\|d_{tt}\|_{L^2}^2+CE_0^\frac12\|\nabla\dot{u}\|_{L^2}^2
+C\|\nabla u\|_{L^4}^4+C\|\nabla u\|_{L^2}^4\nonumber\\
&\quad+C\|\nabla u\|_{L^2}^4\|\nabla d_t\|_{L^2}^2+C\|\nabla u\|_{L^2}^2\|\nabla d_t\|_{L^2}^2
+C\|\nabla d\|_{H^1}^4(\|d_t\|_{L^2}^2+\|\nabla d_t\|_{L^2}^2)\nonumber\\
&\quad+C\|\nabla d\|_{H^1}^2\|\nabla d_t\|_{L^2}^2+C\|\nabla u\|_{L^2}^4\|\nabla^3d\|_{L^2}^2.
\end{align}
This together with \eqref{t3.46} gives that
\begin{align}
&\frac{d}{dt}\Big(\frac{\eta(t)}{2}\|\nabla d_t\|_{L^2}^2\Big)
+\eta(t)\|d_{tt}\|_{L^2}^2-\frac12\eta'(t)\|\nabla d_t\|_{L^2}^2\nonumber\\
&\le C\delta\eta(t)\|d_{tt}\|_{L^2}^2+CE_0^\frac12\eta(t)\|\nabla\dot{u}\|_{L^2}^2
+C\eta(t)\|\nabla u\|_{L^4}^4+C\eta(t)\|\nabla u\|_{L^2}^4\nonumber\\
&\quad+C\eta(t)\|\nabla u\|_{L^2}^4\|\nabla d_t\|_{L^2}^2+C\eta(t)\|\nabla u\|_{L^2}^2\|\nabla d_t\|_{L^2}^2
+C\eta(t)\|\nabla d\|_{H^1}^4(\|d_t\|_{L^2}^2+\|\nabla d_t\|_{L^2}^2)\nonumber\\
&\quad+C\eta(t)\|\nabla d\|_{H^1}^2\|\nabla d_t\|_{L^2}^2+C\eta(t)\|\nabla u\|_{L^2}^4\|\nabla^3d\|_{L^2}^2,
\end{align}
which combined with \eqref{3.59} leads to \eqref{t3.36} after choosing $E_0\le \varepsilon_3$ and $\delta$ sufficiently small.
\end{proof}

\begin{lemma}\label{l34}
Let the assumptions of Proposition \eqref{p31} hold. Then there exists a positive constant $\varepsilon_4$ such that
\begin{align}\label{3.65}
A_2(\sigma(T))+\int_0^{\sigma(T)}\big(\|\sqrt{\rho}\dot{u}\|_{L^2}^2
+\|\nabla d_t\|_{L^2}^2\big)dt\le 3K,
\end{align}
provided that $E_0\le \varepsilon_4$.
\end{lemma}
\begin{proof}[Proof]
In view of Lemma \ref{l28} and \eqref{3.8}, we have
\begin{align}\label{t3.50}
\|\nabla u\|_{L^3}^3&\le C\big(\|\rho\dot{u}\|_{L^2}+\||\nabla d||\nabla^2d|\|_{L^2}\big)^\frac32\big(\|\nabla u\|_{L^2}+\|P-\bar{P}\|_{L^2}\big)^\frac32
+C\big(\|\nabla u\|_{L^2}^3+\|P-\bar{P}\|_{L^3}^3\big)\nonumber\\
&\le \delta\|\rho\dot{u}\|_{L^2}^2+C\|\nabla d\|_{L^6}^2\|\nabla^2d\|_{L^3}^2+C\|\nabla u\|_{L^2}^6+C\|\nabla u\|_{L^2}^4+C\|\nabla u\|_{L^2}^2+CE_0\nonumber\\
&\le  \delta\|\rho\dot{u}\|_{L^2}^2+C\|\nabla d\|_{H^1}^2(\|\nabla^2d\|_{L^2}\|\nabla^3d\|_{L^2}
+\|\nabla^2d\|_{L^2}^2)+C\|\nabla u\|_{L^2}^6\nonumber\\
&\quad+C\|\nabla u\|_{L^2}^4+C\|\nabla u\|_{L^2}^2+CE_0\nonumber\\
&\le \delta\|\rho\dot{u}\|_{L^2}^2+\delta\|\nabla d\|_{H^1}^2\|\nabla^3d\|_{L^2}^2+C\|\nabla d\|_{H^1}^4
+C\|\nabla u\|_{L^2}^6\nonumber\\
&\quad+C\|\nabla u\|_{L^2}^4+C\|\nabla u\|_{L^2}^2+CE_0\nonumber\\
&\le \delta\|\rho\dot{u}\|_{L^2}^2+\delta\|\nabla d\|_{H^1}^2\|\nabla d_t\|_{L^2}^2+C\|\nabla d\|_{H^1}^4+C\|\nabla d\|_{H^1}^6
\nonumber\\
&\quad+C\|\nabla u\|_{L^2}^6+C\|\nabla u\|_{L^2}^4+C\|\nabla u\|_{L^2}^2+CE_0,
\end{align}
due to
\begin{align*}
\|P-\bar{P}\|_{L^3}\le C\|P-\bar{P}\|_{L^\infty}^\frac13\Big(\int|P-\bar{P}|^2dx\Big)^\frac13\le C(\hat{\rho})\|\rho-\bar{\rho}\|_{L^2}^\frac23.
\end{align*}
By \eqref{t3.50}, \eqref{3.8}, \eqref{t3.4}, and Lemma \ref{l33}, one gets that
\begin{align}\label{3.68}
\int_0^{\sigma(T)}\|\nabla u\|_{L^3}^3dt
&\le \delta C(\hat{\rho})\int_0^{\sigma(T)}\|\sqrt{\rho}\dot{u}\|_{L^2}^2dt+C\delta\int_0^{\sigma(T)}\|\nabla d_t\|_{L^2}^2\nonumber\\
& \quad +C\int_0^{\sigma(T)}\big(\|\nabla d\|_{H^1}^4+\|\nabla u\|_{L^2}^4+\|\nabla d\|_{H^1}^6\big)dt+C\int_0^{\sigma(T)}\|\nabla u\|_{L^2}^6dt+CE_0.
\end{align}
Taking $\eta(t)=1$ and integrating \eqref{t3.20} over $[0, t]$ for $0<t\le \sigma(T)$, we deduce from \eqref{3.5},
\eqref{3.6}, and \eqref{3.68} that
\begin{align}\label{t3.52}
&A_2(t)+\int_0^t\big(\|\sqrt{\rho}\dot{u}\|_{L^2}^2+\|\nabla d_t\|_{L^2}^2\big)d\tau\nonumber\\
&\le \int(P-\bar{P}){\rm div}\,udx\Big|_{\tau=0}^t
-\int M(d):\nabla udx\Big|_{\tau=0}^t+C\int_0^t\|\nabla u\|_{L^3}^3d\tau\nonumber\\
&\quad+C\int_0^t\big(\|\nabla d\|_{H^1}^2+\|\nabla d\|_{H^1}^4+\|\nabla u\|_{L^2}^4+\|\nabla d\|_{H^1}^6
+\|\nabla u\|_{L^2}^6\big)d\tau+CE_0\nonumber\\
&\le K+\frac14A_2(t)+CE_0A_2(t)+CE_0A_2^2(t)+CE_0,
\end{align}
due to
\begin{align*}
\|\nabla^2d\|_{L^2}^2\le C\|\Delta d\|_{L^2}^2+C\|\nabla d\|_{L^2}^2.
\end{align*}
We infer from \eqref{t3.52} that
\begin{align}\label{3.70}
&A_2(\sigma(T))+\int_0^{\sigma(T)}\big(\|\sqrt{\rho}\dot{u}\|_{L^2}^2+\|\nabla d_t\|_{L^2}^2\big)dt\nonumber\\
&\le 2K+CE_0A_2^2(t)+CE_0\le \frac52K+CE_0KA_2(\sigma(T)),
\end{align}
provided that $E_0\le\varepsilon_4$ is suitably small. We immediately obtain \eqref{3.65} from \eqref{3.70}.
\end{proof}

\begin{lemma}\label{l35}
Let the assumptions of Proposition \ref{p31} be satisfied. For $\sigma_i\triangleq\sigma(t+1-i)$ with $i$ being an integer satisfying $1\le i\le [T]-1$, then there exists a positive constant $\varepsilon_5$ such that
\begin{align}\label{z3.44}
A_1(T)+\int_{i-1}^{i+1}\sigma_i\big(\|\sqrt{\rho}\dot{u}\|_{L^2}^2+\|\nabla d_t\|_{L^2}^2\big)dt\le E_0^\frac12,
\end{align}
provided that $E_0\le \varepsilon_5$.
\end{lemma}
\begin{remark}
For simplicity, we only prove the case $T>2$. Otherwise, the same thing can be done by choosing a suitably small step size.
\end{remark}
\begin{proof}[Proof]
For integer $i\ (1\le i\le [T]-1)$, taking $\eta(t)=\sigma_i$ and integrating \eqref{t3.20} over $(i-1, i+1]$, we derive
\begin{align}\label{3.72}
&\sup_{i-1\le t\le i+1}\big[\sigma_i\big(\|\nabla u\|_{L^2}^2
+\|\Delta d\|_{L^2}^2\big)\big]+\int_{i-1}^{i+1}\sigma_i\big(\|\sqrt{\rho}\dot{u}\|_{L^2}^2+\|\nabla d_t\|_{L^2}^2\big)dt\nonumber\\
&\le \sigma_i\int(P-\bar{P}){\rm div}\,udx+\sigma_i\int M(d):\nabla udx+C\int_{i-1}^{i+1}\big(\sigma_i+\sigma_i'\big)\|\nabla u\|_{L^2}^2dt\nonumber\\
&\quad+C\int_{i-1}^{i+1}\sigma_i'\big(\|\nabla^2 d\|_{L^2}^4+\|\nabla u\|_{L^2}^4+E_0^2\big)dt
+C\int_{i-1}^{i+1}\sigma_i\|\nabla u\|_{L^3}^3dt\nonumber\\
&\quad+C\int_{i-1}^{i+1}\sigma_i\big(\|\nabla u\|_{L^2}^4+\|\nabla^2d\|_{L^2}^4+\|\nabla u\|_{L^2}^6+\|\nabla^2d\|_{L^2}^6\big)dt
+CE_0\nonumber\\
&\le \frac12\sup_{i-1\le t\le i+1}\big(\sigma_i\|\nabla u\|_{L^2}^2\big)+C\sup_{i-1\le t\le i+1}\big(\sigma_i\|\nabla d\|_{L^4}^4\big)
+C\int_{i-1}^{i+1}\|\nabla u\|_{L^2}^2dt\nonumber\\
&\quad+C\sup_{i-1\le t\le i+1}\big(\|\nabla u\|_{L^2}^2
+\|\Delta d\|_{L^2}^2+\|\nabla d\|_{L^2}^2\big)\int_{i-1}^{i+1}\big(\|\nabla u\|_{L^2}^2+\|\nabla^2d\|_{L^2}^2\big)dt\nonumber\\
&\quad+C\sup_{i-1\le t\le i+1}\big(\|\nabla u\|_{L^2}^4
+\|\Delta d\|_{L^2}^4+\|\nabla d\|_{L^2}^4\big)\int_{i-1}^{i+1}\big(\|\nabla u\|_{L^2}^2+\|\nabla^2d\|_{L^2}^2\big)dt\nonumber\\
&\quad+C\sup_{i-1\le t\le i}\|\nabla u\|_{L^2}^2\int_{i-1}^i\|\nabla u\|_{L^2}^2dt+CE_0+\delta\int_{i-1}^{i+1}\sigma_i\big(\|\sqrt{\rho}\dot{u}\|_{L^2}^2
+\|\nabla d_t\|_{L^2}^2\big)dt\nonumber\\
&\le  \frac12\sup_{i-1\le t\le i+1}\big(\sigma_i\|\nabla u\|_{L^2}^2\big)+CE_0^\frac12\sup_{i-1\le t\le i+1}\big(\sigma_i\|\Delta d\|_{L^2}^2\big)\nonumber\\
&\quad+CE_0\big(A_1(T)+A_2(\sigma(T)+E_0)\big)+CE_0\big(A_1^2(T)+A_2^2(\sigma(T))+E_0^2\big)\nonumber\\
&\quad+\delta\int_{i-1}^{i+1}\sigma_i\big(\|\sqrt{\rho}\dot{u}\|_{L^2}^2
+\|\nabla d_t\|_{L^2}^2\big)dt+CE_0\nonumber\\
&\le  \frac12\sup_{i-1\le t\le i+1}\big(\sigma_i\|\nabla u\|_{L^2}^2\big)+CE_0^\frac12\sup_{i-1\le t\le i+1}\big(\sigma_i\|\Delta d\|_{L^2}^2\big)\nonumber\\
&\quad+\delta\int_{i-1}^{i+1}\sigma_i\big(\|\sqrt{\rho}\dot{u}\|_{L^2}^2
+\|\nabla d_t\|_{L^2}^2\big)dt+CE_0,
\end{align}
where we have used \eqref{3.6}, \eqref{3.8}, and \eqref{t3.15}. Choosing $\delta$ and $E_0$ suitably small, we deduce from \eqref{3.72} that
\begin{align}\label{3.73}
\sup_{0\le t\le \sigma(T)}\big[\sigma\big(\|\nabla u\|_{L^2}^2
+\|\Delta d\|_{L^2}^2\big)\big]\le CE_0\le E_0^\frac12,
\end{align}
due to $\sigma_1(t)=\sigma(t)$ and
\begin{align}\label{3.74}
\sup_{i\le t\le i+1}\big[\sigma_i\big(\|\nabla u\|_{L^2}^2
+\|\Delta d\|_{L^2}^2\big)\big]
+\int_{i-1}^{i+1}\sigma_i\big(\|\sqrt{\rho}\dot{u}\|_{L^2}^2
+\|\nabla d_t\|_{L^2}^2\big)dt\le C_3E_0\le E_0^\frac12,
\end{align}
provided that $E_0\le \varepsilon_5$ is properly small. Note that the constant $C_3$ is independent of $i$.
Thus, the desired \eqref{z3.44} follows from \eqref{3.73} and \eqref{3.74}.
\end{proof}

\begin{lemma}\label{l36}
Let the assumptions of Proposition \ref{p31} be satisfied. Then there exists a positive constant $\varepsilon_6$ such that, for $\sigma(T)\le t_1<t_2\le T$,
\begin{align}\label{3.75}
\sup_{0\le t\le T}\big[\sigma^2\big(\|\sqrt{\rho}\dot{u}\|_{L^2}^2+\|\nabla d_t\|_{L^2}^2\big)\big]\le CE_0^\frac12,
\end{align}
and
\begin{align}\label{3.76}
\int_{t_1}^{t_2}\sigma^2\big(\|\nabla\dot{u}\|_{L^2}^2+\|d_{tt}\|_{L^2}^2\big)dt
\le CE_0^\frac12+CE_0(t_2-t_1),
\end{align}
provided that $E_0\le \varepsilon_6$.
\end{lemma}
\begin{proof}[Proof]
1. We get from Lemma \ref{l28} that
\begin{align*}
\|\nabla u\|_{L^4}^4\le C\big(\|\rho\dot{u}\|_{L^2}+\||\nabla d||\nabla^2d|\|_{L^2}\big)^3\big(\|\nabla u\|_{L^2}+\|P-\bar{P}\|_{L^2}\big)+C\big(\|\nabla u\|_{L^2}^4+\|P-\bar{P}\|_{L^4}^4\big),
\end{align*}
which implies that
\begin{align}
\int_{i-1}^{i+1}\sigma_i^2\|\nabla u\|_{L^4}^4dt
&\le C\int_{i-1}^{i+1}\sigma_i^2\Big(\|\nabla u\|_{L^2}+E_0^\frac12\Big)\big(\|\sqrt{\rho}\dot{u}\|_{L^2}^3
+\|\nabla d\|_{L^3}^3\|\nabla^2d\|_{L^6}^3\big)dt\nonumber\\
&\quad+C\int_{i-1}^{i+1}\sigma_i^2\big(\|\nabla u\|_{L^2}^4\big)dt+CE_0\nonumber\\
&\le C\sup_{i-1\le t\le i+1}\Big(\|\nabla u\|_{L^2}+E_0^\frac12\Big)\int_{i-1}^{i+1}\sigma_i^2\big(
\|\sqrt{\rho}\dot{u}\|_{L^2}^3
+\|\nabla^3d\|_{L^2}^3\big)dt\nonumber\\
&\quad+C\int_{i-1}^{i+1}\sigma_i^2(\|\nabla u\|_{L^2}^4+\|\nabla d\|_{H^1}^4)dt+CE_0\nonumber\\
&\le C\sup_{i-1\le t\le i+1}\Big(\|\nabla u\|_{L^2}+E_0^\frac12\Big)\int_{i-1}^{i+1}\sigma_i^2\big(
\|\sqrt{\rho}\dot{u}\|_{L^2}^3
+\|\nabla d_t\|_{L^2}^3\big)dt\nonumber\\
&\quad+C\int_{i-1}^{i+1}\sigma_i^2\big(\|\nabla u\|_{L^2}^4+\|\nabla^2d\|_{L^2}^4\big)dt+CE_0,
\end{align}
\begin{align}
\int_0^T\|\nabla u\|_{L^2}^4dt&=\int_0^{\sigma(T)}\|\nabla u\|_{L^2}^4dt+\int_{\sigma(T)}^T\|\nabla u\|_{L^2}^4dt\nonumber\\
&\le \sup_{0\le t\le \sigma(T)}\|\nabla u\|_{L^2}^2\int_0^{\sigma(T)}\|\nabla u\|_{L^2}^2dt
+\sup_{\sigma(T)\le t\le T}\big(\sigma\|\nabla u\|_{L^2}^2\big)\int_{\sigma(T)}^T\|\nabla u\|_{L^2}^2dt\nonumber\\
&\le C(K)E_0,
\end{align}
and
\begin{align}
\|P-\bar{P}\|_{L^4}\le C\|P-\bar{P}\|_{L^\infty}^\frac12\Big(\int|P-\bar{P}|^2dx\Big)^\frac14\le C(\hat{\rho})\|\rho-\bar{\rho}\|_{L^2}^\frac12.
\end{align}
For any integer $1\le i\le [T]-1$, integrating \eqref{t3.36} with $\eta(t)=\sigma_i^2$ over $(i-1, i+1]$,
we deduce from \eqref{3.6}, \eqref{3.9}, \eqref{3.65}, and Young's inequality that
\begin{align}\label{3.78}
&\sup_{i-1\le t\le i+1}\big[\sigma_i^2\big(\|\sqrt{\rho}\dot{u}\|_{L^2}^2
+\|\nabla d_t\|_{L^2}^2\big)\big]+\int_{i-1}^{i+1}\sigma_i^2\big(\|\nabla\dot{u}\|_{L^2}^2
+\|d_{tt}\|_{L^2}^2\big)dt\nonumber\\
&\le -\int_{\partial\Omega}\sigma_i^2(u\cdot\nabla n\cdot u)FdS\Big|_{i-1}^{i+1}
+C\int_{i-1}^{i+1}\sigma_i^2\|\nabla u\|_{L^4}^4dt\nonumber\\
&\quad+C\int_{i-1}^{i+1}\sigma_i\sigma_i'\big(\|\sqrt{\rho}\dot{u}\|_{L^2}^2
+\|\nabla d_t\|_{L^2}^2+\|\nabla u\|_{L^2}^2
+\|\nabla^2d\|_{L^2}^2+\|\nabla u\|_{L^2}^4+\|\nabla d\|_{H^1}^4\big)dt\nonumber\\
&\quad+C\int_{i-1}^{i+1}\sigma_i^2\big(\|\sqrt{\rho}\dot{u}\|_{L^2}^2\|\nabla u\|_{L^2}^2
+\|\nabla u\|_{L^2}^4\|\nabla d_t\|_{L^2}^2+\|\nabla u\|_{L^2}^6+\|\nabla d\|_{H^1}^6\big)dt\nonumber\\
&\quad+C\int_{i-1}^{i+1}\sigma_i^2\big(\|\nabla u\|_{L^2}^4+\|\nabla d\|_{H^1}^4+\|\nabla u\|_{L^2}^2
+\delta\|\nabla^2d\|_{L^2}\|\nabla d_t\|_{L^2}^3\big)dt\nonumber\\
&\quad+C\int_{i-1}^{i+1}\sigma_i^2\big(\|\nabla d\|_{H^1}^2\|\nabla d_t\|_{L^2}^2+\|\nabla d\|_{H^1}^4\|d_t\|_{L^2}^2+\|\nabla u\|_{L^2}^2\|\nabla d_t\|_{L^2}^2\big)dt\nonumber\\
&\quad+C\int_{i-1}^{i+1}\sigma_i^2(\|\nabla u\|_{L^2}^4+\|\nabla d\|_{H^1}^4)\|\nabla d_t\|_{L^2}^2dt
\nonumber\\
&\le C\sup_{i-1\le t\le i+1}\big(\sigma_i^2\|\sqrt{\rho}\dot{u}\|_{L^2}^2\big)
\int_{i-1}^{i+1}\big(\|\nabla u\|_{L^2}^4
+\|\nabla u\|_{L^2}^2\big)dt\nonumber\\
&\quad+C\sup_{i-1\le t\le i+1}\big(\sigma_i^2\|\nabla d_t\|_{L^2}^2\big)
\int_{i-1}^{i+1}\big(\|\nabla u\|_{L^2}^4
+\|\nabla u\|_{L^2}^2+\|\nabla d\|_{H^1}^2\big)dt\nonumber\\
&\quad+C\sup_{i-1\le t\le i+1}\Big(\|\nabla u\|_{L^2}+\|\Delta d\|_{L^2}+E_0^\frac12+\delta\Big)\int_{i-1}^{i+1}\sigma_i^2\big(
\|\sqrt{\rho}\dot{u}\|_{L^2}^3
+\|\nabla d_t\|_{L^2}^3\big)dt\nonumber\\
&\quad+C\sup_{i-1\le t\le i+1}(\sigma_i\|\nabla d\|_{H^1}^4)\int_{i-1}^{i+1}\sigma_i\|d_t\|_{L^2}^2dt
+C\int_{i-1}^{i+1}\sigma_i^2\big(\|\nabla u\|_{L^2}^4+\|\nabla^2d\|_{L^2}^4\big)dt\nonumber\\
&\quad+\frac14\sup_{i-1\le t\le i+1}\big[\sigma_i^2\big(\|\sqrt{\rho}\dot{u}\|_{L^2}^2
+\|\nabla d_t\|_{L^2}^2\big)\big]+CE_0\nonumber\\
&\le \frac14\sup_{i-1\le t\le i+1}\big[\sigma_i^2\big(\|\sqrt{\rho}\dot{u}\|_{L^2}^2
+\|\nabla d_t\|_{L^2}^2\big)\big]\nonumber\\
&\quad+C\sup_{i-1\le t\le i+1}\big[\sigma_i\big(\|\sqrt{\rho}\dot{u}\|_{L^2}
+\|\nabla d_t\|_{L^2}\big)\big]
\int_{i-1}^{i+1}\sigma_i\big(\|\sqrt{\rho}\dot{u}\|_{L^2}^2
+\|\nabla d_t\|_{L^2}^2\big)dt\nonumber\\
&\quad+CE_0\sup_{i-1\le t\le i+1}\big[\sigma_i^2\big(\|\sqrt{\rho}\dot{u}\|_{L^2}^2
+\|\nabla d_t\|_{L^2}^2\big)\big]+CE_0^\frac12\nonumber\\
&\le\frac14\sup_{i-1\le t\le i+1}\big[\sigma_i^2\big(\|\sqrt{\rho}\dot{u}\|_{L^2}^2
+\|\nabla d_t\|_{L^2}^2\big)\big]+CE_0^\frac12\sup_{i-1\le t\le i+1}\big[\sigma_i\big(\|\sqrt{\rho}\dot{u}\|_{L^2}
+\|\nabla d_t\|_{L^2}\big)\big]\nonumber\\
&\quad
+CE_0\sup_{i-1\le t\le i+1}\big[\sigma_i^2\big(\|\sqrt{\rho}\dot{u}\|_{L^2}^2
+\|\nabla d_t\|_{L^2}^2\big)\big]+CE_0^\frac12\nonumber\\
&\le \frac14\sup_{i-1\le t\le i+1}\big[\sigma_i^2\big(\|\sqrt{\rho}\dot{u}\|_{L^2}^2
+\|\nabla d_t\|_{L^2}^2\big)\big]+CE_0^\frac12,
\end{align}
where we have used
\begin{align*}
\int_{\partial\Omega}(u\cdot\nabla n\cdot u)FdS
&\le C\||u|^2|F|\|_{W^{1,1}}\le C\|\nabla u\|_{L^2}^2\|F\|_{H^1}\nonumber\\
&\le \frac{1}{4}(\|\sqrt{\rho}\dot{u}\|_{L^2}^2+\|\nabla d_t\|_{L^2}^2)
+C\big(\|\nabla u\|_{L^2}^4+\|\nabla d\|_{H^1}^4+\|\nabla d\|_{H^1}^2\big).
\end{align*}
According to \eqref{3.78}, we get that
\begin{align}\label{3.81}
\sup_{0\le t\le \sigma(T)}\big[\sigma_i^2\big(\|\sqrt{\rho}\dot{u}\|_{L^2}^2
+\|\nabla d_t\|_{L^2}^2\big)\big]\le CE_0^\frac12,
\end{align}
and
\begin{align}\label{3.82}
\sup_{i\le t\le i+1}\big(\|\sqrt{\rho}\dot{u}\|_{L^2}^2
+\|\nabla d_t\|_{L^2}^2\big)\le CE_0^\frac12.
\end{align}
Hence, we deduce \eqref{3.75} from \eqref{3.6}, \eqref{3.81}, and \eqref{3.82}.

2. We integrate \eqref{t3.20} over $[t_1, t_2]\subseteq[\sigma(T), T]$ and take $\eta(t)=\sigma$ to obtain, from \eqref{3.6}
and \eqref{3.9}, that
\begin{align}\label{3.83}
\int_{t_1}^{t_2}\sigma\big(\|\sqrt{\rho}\dot{u}\|_{L^2}^2+\|\nabla d_t\|_{L^2}^2\big)dt
&\le C(E_0+A_1(T))+CE_0(t_2-t_1)+C\int_{t_1}^{t_2}\sigma\big(\|\nabla u\|_{L^2}^6+\|\nabla d\|_{H^1}^6\big)dt\nonumber\\
&\quad+C\int_{t_1}^{t_2}\sigma\big(\|\nabla u\|_{L^2}^4+\|\nabla d\|_{H^1}^4\big)dt\nonumber\\
&\le CE_0^\frac12+CE_0(t_2-t_1)+C\int_{t_1}^{t_2}\big(\|\nabla u\|_{L^2}^2
+\|\nabla d\|_{H^1}^2\big)dt\nonumber\\
&\le CE_0^\frac12+CE_0(t_2-t_1).
\end{align}
Similarly to \eqref{3.78}, integrating \eqref{t3.36} over $[t_1, t_2]$ and taking $\eta=\sigma^2$, we find that
\begin{align}\label{3.84}
&\int_{t_1}^{t_2}\sigma^2\big(\|\nabla\dot{u}\|_{L^2}^2+\|d_{tt}\|_{L^2}^2\big)dt\nonumber\\
&\le C\sup_{t_1\le t\le t_2}\big[\sigma\big(\|\sqrt{\rho}\dot{u}\|_{L^2}
+\|\nabla d_t\|_{L^2}\big)\big]
\int_{t_1}^{t_2}\sigma\big(\|\sqrt{\rho}\dot{u}\|_{L^2}^2
+\|\nabla d_t\|_{L^2}^2\big)dt
+CE_0^\frac12\nonumber\\
&\le CE_0^\frac12+CE_0^\frac16\int_{t_1}^{t_2}
\sigma\big(\|\sqrt{\rho}\dot{u}\|_{L^2}^2+\|\nabla d_t\|_{L^2}^2\big)dt\nonumber\\
&\le CE_0^\frac12+CE_0(t_2-t_1),
\end{align}
owing to \eqref{3.6}, \eqref{3.9}, \eqref{3.75}, and \eqref{3.83}. The conclusion follows.
\end{proof}

We still need the following result before showing the upper bounds of the density.
\begin{lemma}\label{l37}
Let the assumption of Proposition \ref{p31} be satisfied. Then there exists a positive constant $\varepsilon_7$ such that
\begin{align}\label{3.85}
\sup_{0\le t\le \sigma(T)}\big[\sigma\big(\|\sqrt{\rho}\dot{u}\|_{L^2}^2+
\|\nabla d_t\|_{L^2}^2\big)\big]+\int_0^{\sigma(T)}\sigma
\big(\|\nabla\dot{u}\|_{L^2}^2+\|d_{tt}\|_{L^2}^2\big)dt\le C,
\end{align}
provided that $E_0\le \varepsilon_7$.
\end{lemma}
\begin{proof}[Proof]
Taking $\eta(t)=\sigma$ and integrating \eqref{t3.36} over $[0, \sigma(T)]$,
we get from \eqref{3.6}, \eqref{3.9}, \eqref{3.65}, and Young's inequality that
\begin{align*}
&\sup_{0\le t\le \sigma(T)}\big[\sigma\big(\|\sqrt{\rho}\dot{u}\|_{L^2}^2
+\|\nabla d_t\|_{L^2}^2\big)\big]
+\int_0^{\sigma(T)}\sigma(\|\nabla\dot{u}\|_{L^2}^2+\|d_{tt}\|_{L^2}^2)dt\nonumber\\
&\le C\int_0^{\sigma(T)}\sigma'\big(\|\sqrt{\rho}\dot{u}\|_{L^2}^2+\|\nabla d_t\|_{L^2}^2\big)dt+CE_0^\frac12\nonumber\\
&\quad+C\int_0^{\sigma(T)}\sigma\big(\|\nabla u\|_{L^2}+\|\Delta d\|_{L^2}+E_0^\frac12+\delta\big)
\big(\|\sqrt{\rho}\dot{u}\|_{L^2}^3
+\|\nabla d_t\|_{L^2}^3\big)dt\nonumber\\
&\le CE_0^\frac14\sup_{0\le t\le \sigma(T)}\big[\sigma^\frac12\big(\|\sqrt{\rho}\dot{u}\|_{L^2}
+\|\nabla d_t\|_{L^2}\big)\big]\int_0^{\sigma(T)}
\big(\|\sqrt{\rho}\dot{u}\|_{L^2}^2
+\|\nabla d_t\|_{L^2}^2\big)dt+C(K)\nonumber\\
&\le C(K)+C(K)\sup_{0\le t\le \sigma(T)}\big[\sigma^\frac12\big(\|\sqrt{\rho}\dot{u}\|_{L^2}^2
+\|\nabla d_t\|_{L^2}\big)\big]\nonumber\\
&\le \frac12\sup_{0\le t\le \sigma(T)}\big[\sigma\big(\|\sqrt{\rho}\dot{u}\|_{L^2}^2
+\|\nabla d_t\|_{L^2}^2\big)\big]+C,
\end{align*}
from which, the conclusion follows.
\end{proof}

With Lemmas \ref{l36} and \ref{l37} at hand, we derive the uniform upper bounds of the density, which is the key to obtain
all the higher-order estimates and thus to extend the classical solution globally.
\begin{lemma}\label{l38}
There exists a positive constant $\varepsilon$ as in Theorem \ref{thm1} such that if $(\rho, u, d)$ is a strong solution of \eqref{a1}--\eqref{a6} in $\Omega\times(0, T]$ satisfying \eqref{3.6}, then
\begin{align}\label{3.87}
\sup_{0\le t\le T}\|\rho(t)\|_{L^\infty}\le \frac74\hat{\rho}
\end{align}
provided that $E_0\le\varepsilon\triangleq\min\{1, \varepsilon_2, \varepsilon_3, \varepsilon_4, \varepsilon_5, \varepsilon_6,
\varepsilon_7\}$.
\end{lemma}
\begin{proof}[Proof]
1. We rewrite $\eqref{a1}_1$ as
\begin{align}
D_t\rho=g(\rho)+b'(t),
\end{align}
where
\begin{align*}
D_t\rho=\rho_t+u\cdot\nabla\rho, \quad g(\rho)=-\frac{\rho(P-\bar{P})}{2\mu+\lambda},
\quad b(t)=-\frac{1}{2\mu+\lambda}\int_0^t\rho Fd\tau.
\end{align*}
For $t\in [0, \sigma(T)]$, we deduce from H\"older's inequality, Lemma \ref{l28}, and \eqref{3.65} that, for
$0\le t_1<t_2\le \sigma(T)$,
\begin{align}
&|b(t_2)-b(t_1)|\nonumber\\
&\le C(\hat{\rho})\int_0^{\sigma(T)}\|F\|_{L^\infty}dt\nonumber\\
&\le C\int_0^{\sigma(T)}\|F\|_{L^2}^\frac14\|\nabla F\|_{L^6}^\frac34dt+C\int_0^{\sigma(T)}\|F\|_{L^2}dt
\nonumber\\
&\le C\int_0^{\sigma(T)}\Big(
\|\nabla\dot{u}\|_{L^2}+\|\nabla u\|_{L^2}^2
+\|\nabla^2d\|_{L^2}^\frac12\|\nabla^3d\|_{L^2}^\frac32
+\|\nabla^2d\|_{L^2}^\frac32\|\nabla^3d\|_{L^2}^\frac12+\|\nabla d\|_{L^2}\|\nabla^2d\|_{L^2}\nonumber\\
&\quad+\|\nabla d\|_{L^2}\|\nabla^3d\|_{L^2}\Big)^\frac34\Big(\|\nabla u\|_{L^2}^\frac14
+\|P-\bar{P}\|_{L^2}^\frac14\Big)dt+C\int_0^{\sigma(T)}(\|\nabla u\|_{L^2}+\|P-\bar{P}\|_{L^2})dt\nonumber\\
&\le C\int_0^{\sigma(T)}\Big(
\|\nabla\dot{u}\|_{L^2}+\|\nabla u\|_{L^2}^2
+\|\nabla^2d\|_{L^2}^\frac12\|\nabla d_t\|_{L^2}^\frac32
+\|\nabla^2d\|_{L^2}^\frac32\|\nabla d_t\|_{L^2}^\frac12+\|\nabla d\|_{H^1}^2\nonumber\\
&\quad+\|\nabla d_t\|_{L^2}\Big)^\frac34\Big(\|\nabla u\|_{L^2}^\frac14
+\|P-\bar{P}\|_{L^2}^\frac14\Big)dt+C\int_0^{\sigma(T)}(\|\nabla u\|_{L^2}+\|P-\bar{P}\|_{L^2})dt\nonumber\\
&\le CE_0^\frac18
+C\int_0^{\sigma(T)}\Big(\sigma^{-\frac12}(\sigma^\frac12\|\nabla u\|_{L^2})^\frac14+E_0^\frac18\sigma^{-\frac38}\Big)
\big(\sigma\|\nabla\dot{u}\|_{L^2}^2\big)^\frac38dt\nonumber\\
&\quad+C\int_0^{\sigma(T)}\sigma^{-\frac18}\big(\sigma\|\nabla u\|_{L^2}^2\big)^\frac18dt+C\int_0^{\sigma(T)}\Big(\sigma^{-\frac18}(\sigma\|\nabla u\|_{L^2}^2)^\frac18+E_0^\frac18\Big)\|\nabla d_t\|_{L^2}^\frac98dt\nonumber\\
&\quad+C\int_0^{\sigma(T)}\Big(\sigma^{-\frac18}(\sigma\|\nabla u\|_{L^2}^2)^\frac18
+E_0^\frac18\Big)\big(\|\nabla d_t\|_{L^2}\big)^\frac38dt\nonumber\\
&\quad+CE_0^\frac38\int_0^{\sigma(T)}\Big(\sigma^{-\frac18}(\sigma\|\nabla u\|_{L^2}^2)^\frac18
+E_0^\frac18\Big)\big(\|\nabla d_t\|_{L^2}\big)^\frac34dt\nonumber\\
&\le C\Big[E_0^\frac{1}{16}\Big(\int_0^{\sigma(T)}
\sigma^{-\frac45}dt\Big)^\frac58
+E_0^\frac18\Big(\int_0^{\sigma(T)}\sigma^{-\frac35}dt\Big)^\frac58\Big]
\Big[\int_0^{\sigma(T)}\sigma\|\nabla\dot{u}\|_{L^2}^2dt\Big]^\frac38\nonumber\\
&\quad+C\left[E_0^\frac18+E_0^\frac{1}{16}
\Big(\int_0^{\sigma(T)}\sigma^{-\frac{2}{7}}dt\Big)^\frac{7}{16}\right]
\Big(\int_0^{\sigma(T)}\|\nabla d_t\|_{L^2}^2\Big)^\frac{9}{16}\nonumber\\
&\quad+C\left[E_0^\frac18+E_0^\frac{1}{16}
\Big(\int_0^{\sigma(T)}\sigma^{-\frac{2}{13}}dt\Big)^\frac{13}{16}\right]
\Big(\int_0^{\sigma(T)}\|\nabla d_t\|_{L^2}^2\Big)^\frac{3}{16}\nonumber\\
&\quad+C\left[E_0^\frac18+E_0^\frac{1}{16}
\Big(\int_0^{\sigma(T)}\sigma^{-\frac15}dt\Big)^\frac{5}{8}\right]
\Big(\int_0^{\sigma(T)}\|\nabla d_t\|_{L^2}^2\Big)^\frac{3}{8}+CE_0^\frac18\nonumber\\
&\le CE_0^\frac{1}{16},
\end{align}
provided that $E_0\le \varepsilon$. Thus, for $t\in [0, \sigma(T)]$, we can choose $N_0$ and $N_1$ in Lemma \ref{l210} as
\begin{align}
N_1=0, \quad N_0=CE_0^\frac{1}{16},
\end{align}
and
$\xi_0=\hat{\rho}$. Then, one has
\begin{align}
g(\xi)=-\frac{\xi}{2\mu+\lambda}(\xi^\gamma-\bar{\rho}^\gamma)\le -N_1=0 \quad {\rm for~all}~\xi\ge \xi_0=\hat{\rho},
\end{align}
which together with Lemma \ref{l210} implies that
\begin{align}\label{3.92}
\sup_{0\le t\le \sigma(T)}\|\rho\|_{L^\infty}\le \max\{\hat{\rho}, \xi_0\}+N_0\le \hat{\rho}+CE_0^\frac{1}{16}\le \frac{3\hat{\rho}}{2},
\end{align}
provided that $E_0\le \varepsilon$.

2. For $t\in [\sigma(T), T]$, we derive from Lemma \ref{l28}, \eqref{3.6}, \eqref{3.8}, \eqref{3.9}, \eqref{t3.34},
\eqref{z3.44}, and \eqref{3.76} that, for $\sigma(T)\le t_1<t_2\le T$,
\begin{align}
|b(t_2)-b(t_1)|&\le C(\hat{\rho})\int_{t_1}^{t_2}\|F\|_{L^\infty}dt\nonumber\\
&\le C\int_{t_1}^{t_2}\|F\|_{L^\infty}^\frac83dt
+\frac{a}{4\mu+2\lambda}(t_2-t_1)\nonumber\\
&\le C\int_{t_1}^{t_2}\|F\|_{L^2}^\frac23\|\nabla F\|_{L^6}^2dt+\int_{t_1}^{t_2}\|F\|_{L^2}^\frac83dt
+\frac{a}{4\mu+2\lambda}(t_2-t_1)\nonumber\\
&\le C\int_{t_1}^{t_2}\Big(
\|\nabla\dot{u}\|_{L^2}^2+\|\nabla u\|_{L^2}^2
+\|\nabla^2d\|_{L^2}\|\nabla^3d\|_{L^2}^3
+\|\nabla^2d\|_{L^2}^3\|\nabla^3d\|_{L^2}\nonumber\\
&\quad+\|\nabla d\|_{L^2}\|\nabla^2d\|_{L^2}+\|\nabla d\|_{L^2}\|\nabla^3d\|_{L^2}
\Big)\Big(\|\nabla u\|_{L^2}^\frac23
+\|P-\bar{P}\|_{L^2}^\frac23\Big)dt\nonumber\\
&\quad+C\int_{t_1}^{t_2}\Big(\|\nabla u\|_{L^2}^\frac83+\|P-\bar{P}\|_{L^2}^\frac83\Big)dt
+\frac{a}{4\mu+2\lambda}(t_2-t_1)\nonumber\\
&\le \Big(\frac{a}{4\mu+2\lambda}+CE_0^\frac23\Big)(t_2-t_1)
+CE_0^\frac16\int_{t_1}^{t_2}\|\nabla\dot{u}\|_{L^2}^2dt\nonumber\\
&\quad+CE_0^\frac16\Big(\int_{t_1}^{t_2}\|\nabla^2d\|_{L^2}^2\Big)^\frac12
\Big(\int_{t_1}^{t_2}\|\nabla d_t\|_{L^2}^2dt\Big)^\frac12+CE_0^\frac16\int_{t_1}^{t_2}\big(\|\nabla^2d\|_{H^1}^2+\|\nabla u\|_{L^2}^2\big)dt\nonumber\\
&\quad+CE_0^\frac16\sup_{t_1\le t\le t_2}\|\nabla d_t\|_{L^2}\int_{t_1}^{t_2}\|\nabla d_t\|_{L^2}^2dt
+CE_0^\frac16\int_{t_1}^{t_2}\|\nabla d_t\|_{L^2}^2dt\nonumber\\
&\le \Big(\frac{a}{4\mu+2\lambda}+CE_0^\frac23+CE_0^\frac{17}{12}+E_0^\frac76\Big)(t_2-t_1)
+CE_0^\frac16\nonumber\\
&\le \frac{a}{2\mu+\lambda}(t_2-t_1)+CE_0^\frac16,
\end{align}
provided that $E_0\le \varepsilon$. Thus, for $t\in [\sigma(T), T]$, we can choose $N_0$, $N_1$, and $\xi_0$ in Lemma \ref{l210} as follows:
\begin{align*}
N_0=CE_0^\frac16, \quad N_1=\frac{a}{2\mu+\lambda}, \quad \xi_0=\frac{3\hat{\rho}}{2}.
\end{align*}
Since for all $\xi\ge \xi_0=\frac{3\hat{\rho}}{2}>\bar{\rho}+1$,
\begin{align}
g(\xi)=-\frac{\xi}{2}(\xi^\gamma-\bar{\rho}^\gamma)\le -N_1 \quad {\rm for~all}~\xi\ge \xi_0=\frac{3\hat{\rho}}{2}.
\end{align}
Thus, due to Lemma \ref{l210}, we arrive at
\begin{align}\label{3.95}
\sup_{\sigma(T)\le t\le T}\|\rho\|_{L^\infty}\le \frac{3\hat{\rho}}{2}+N_0\le \frac{3\hat{\rho}}{2}+CE_0^\frac16\le \frac{7\hat{\rho}}{4},
\end{align}
provided that $E_0\le \varepsilon$. The combination of \eqref{3.92} and \eqref{3.95}, we obtain \eqref{3.87}.
\end{proof}

Now we are ready to prove Proposition \ref{p31}.
\begin{proof}[Proof of Proposition \ref{p31}]
Proposition \ref{p31} follows from Lemma \ref{l34}, Lemma \ref{l35}, and Lemma \ref{l38}.
\end{proof}

\subsection{Higher-order estimates}\label{sec3.2}

In this subsection, we establish the time-dependent higher-order estimates of solutions, which are necessary for the global existence of
strong solutions. In what follows, we denote by $C$ or $C_i\ (i=1, 2,\ldots)$ the various positive constants, which may depend on the initial data, $\mu$, $\lambda$, $\gamma$, $a$, $\hat{\rho}$, $\Omega$, $M_1$, $M_2$,  $\bar{\rho}$, and $T$ as well.

\begin{lemma}\label{l30}
Under the conditions of Theorem \ref{thm1}, it holds that
\begin{align}
&\sup_{0\le t\le T}\|\nabla\rho\|_{L^q}+\int_0^T\|\nabla u\|_{L^\infty}
dt\le C,\label{z3.72}\\
&\sup_{0\le t\le T}\big[t\big(\|\sqrt{\rho}\dot{u}\|_{L^2}^2+
\|\nabla d_t\|_{L^2}^2\big)\big]+\int_0^Tt\big(\|\nabla\dot{u}\|_{L^2}^2+\|d_{tt}\|_{L^2}^2\big)dt\le C,\label{z3.71}\\
&\sup_{0\le t\le T}\big[t\big(\|u\|_{H^2}^2+\|\nabla^3d\|_{L^2}^2\big)\big]
+\int_0^Tt\big(\|\nabla^2d_t\|_{L^2}^2+\|\nabla^4d\|_{L^2}^2\big)dt\le C.\label{t3.80}
\end{align}
\end{lemma}
\begin{proof}[Proof]
First, based on the Beale-Kato-Majda type
inequality (see Lemma \ref{l211}), we can obtain \eqref{z3.72}. Since the argument is similar to \cite[Lemma 4.1]{CJ21}, we omit the details for simplicity.

By virtue of \eqref{3.85}, it is easy to check that
\begin{align*}
\sup_{0\le t\le \sigma(T)}\big[t(\|\sqrt{\rho}\dot{u}\|_{L^2}^2+
\|\nabla d_t\|_{L^2}^2)\big]
+\int_0^{\sigma(T)}t(\|\nabla\dot{u}\|_{L^2}^2+\|d_{tt}\|_{L^2}^2)dt\le C,
\end{align*}
which together with \eqref{3.76} leads to \eqref{z3.71}.

Next, by Lemma \ref{l24}, \eqref{2.11}, \eqref{2.7}, and \eqref{2.13}, we obtain that
\begin{align*}
\|\nabla^2u\|_{L^2}&\le C\big(\|\divv u\|_{H^1}+\|\curl u\|_{H^1}\big)\nonumber\\
&\le C\big(\|F+P-\bar{P}\|_{H^1}+\|\curl u\|_{H^1}\big)\nonumber\\
&\le C\big(\|\rho\dot{u}\|_{L^2}+\||\nabla d||\nabla^2d|\|_{L^2}+\|\nabla P\|_{L^2}+\|P-\bar{P}\|_{L^2}+\|\nabla u\|_{L^2}\big).
\end{align*}
This along with \eqref{z3.71}, \eqref{t3.34}, \eqref{3.8}, \eqref{t3.6}, and \eqref{3.6} yields that
\begin{align}\label{z3.83}
\sup_{0\le t\le T}\big[t\big(\|u\|_{H^2}^2+\|\nabla^3d\|_{L^2}^2\big)\big]
\le C.
\end{align}

Furthermore, taking the operator $\nabla$ to \eqref{z3.31}, we get that
\begin{align}\label{z3.84}
-\nabla^2\Delta d=\nabla^2(|\nabla d|^2d-u\cdot\nabla d-d_t).
\end{align}
Applying the standard $L^2$-estimate to \eqref{z3.84}, we derive that
\begin{align*}
\|\nabla^4d\|_{L^2}^2&\le C\big(\|\nabla^2d_t\|_{L^2}^2+\|\nabla^2(u\cdot\nabla d)\|_{L^2}^2
+\|\nabla^2(|\nabla d|^2d)\|_{L^2}^2\big)+C\|\nabla d\|_{H^1}^2\nonumber\\
&\le C\|\nabla^2d_t\|_{L^2}^2+C\|u\|_{L^\infty}^2\|\nabla^3d\|_{L^2}^2
+C\|\nabla d\|_{L^\infty}^2\|\nabla^2u\|_{L^2}^2+C\|\nabla u\|_{L^6}^2\|\nabla^2d\|_{L^3}^2\nonumber\\
&\quad+C\|\nabla^2d\|_{L^4}^4+C\|\nabla d\|_{L^6}^2\|\nabla^3d\|_{L^3}^2
+C\|\nabla d\|_{L^2}^2\|\nabla^2d\|_{L^2}^2\|\nabla^2d\|_{L^6}^2+C\nonumber\\
&\le C\|\nabla^2d_t\|_{L^2}^2+C\|\nabla u\|_{H^1}^2\|\nabla^2d\|_{H^1}^2+C\|\nabla^2d\|_{L^2}\|\nabla^3d\|_{L^2}^3
+C\|\nabla^3d\|_{L^2}^2\nonumber\\
&\quad+C\|\nabla d\|_{H^1}^2\big(\|\nabla^3d\|_{L^2}\|\nabla^4d\|_{L^2}+\|\nabla^3d\|_{L^2}^2\big)
+C\nonumber\\
&\le \frac12\|\nabla^4d\|_{L^2}^2+C\|\nabla d\|_{H^1}^2\|\nabla^3d\|_{L^2}^2+C\|\nabla u\|_{H^1}^2\|\nabla^2d\|_{H^1}^2
+C\|\nabla^3d\|_{L^2}^2\nonumber\\
&\quad+C\|\nabla^2d\|_{L^2}\|\nabla^3d\|_{L^2}^3+C,
\end{align*}
which leads to
\begin{align*}
\|\nabla^4d\|_{L^2}^2&\le C\|\nabla d\|_{H^1}^2\|\nabla^3d\|_{L^2}^2+C\|\nabla u\|_{H^1}^2\|\nabla^2d\|_{H^1}^2
+C\|\nabla^3d\|_{L^2}^2\nonumber\\
&\quad+C\|\nabla^2d\|_{L^2}\|\nabla^3d\|_{L^2}^3+C.
\end{align*}
This together with \eqref{z3.83}, \eqref{z3.71}, \eqref{t3.34}, \eqref{3.6}, \eqref{3.8}, \eqref{3.65},
 and \eqref{3.44} implies \eqref{t3.80}.
\end{proof}

\begin{lemma}\label{l3.10}
Under the conditions of Theorem \ref{thm1}, it holds that
\begin{align*}
\sup_{0\le t\le T}\big(t\|\sqrt{\rho}u_t\|_{L^2}^2\big)+\int_0^Tt\|\nabla u_t\|_{L^2}^2dt\le C.
\end{align*}
\end{lemma}
\begin{proof}[Proof]
By Lemma \ref{l30}, Proposition \ref{p31}, \eqref{2.3}, and Sobolev's inequality, we deduce that
\begin{align*}
t\|\sqrt{\rho}u_t\|_{L^2}^2&\le t\|\sqrt{\rho}\dot{u}\|_{L^2}^2+t\|\sqrt{\rho}u\cdot\nabla u\|_{L^2}^2\nonumber\\
&\le C(T)+Ct\|u\|_{L^6}^2\|\nabla u\|_{L^4}^2\nonumber\\
&\le C(T)+C\|\nabla u\|_{L^2}^2\big(t\|u\|_{H^2}^2\big)\nonumber\\
&\le C,
\end{align*}
and
\begin{align*}
\int_0^Tt\|\nabla u_t\|_{L^2}^2dt
&\le \int_0^Tt\|\nabla\dot{u}\|_{L^2}^2dt
+\int_0^Tt\|\nabla(u\cdot\nabla u)\|_{L^2}^2dt\nonumber\\
&\le C(T)+\int_0^Tt\big(\|\nabla u\|_{L^4}^4+\|u\|_{L^\infty}^2\|\nabla^2u\|_{L^2}^2\big)dt\nonumber\\
&\le C(T)+C\int_0^Tt\|\nabla u\|_{H^1}^2\|u\|_{H^2}^2dt\nonumber\\
&\le C.
\end{align*}
This finishes the proof of Lemma \ref{l3.10}.
\end{proof}

\section{Proof of Theorem \ref{thm1}}\label{sec4}
With all the \textit{a priori} estimates in Section \ref{sec3} at hand, we are going to prove the main result of the paper.

\begin{proof}[Proof of Theorem \ref{thm1}]
By Lemma \ref{l22}, there exists a $T_*>0$ such that the system \eqref{a1}--\eqref{a6} has a unique classical solution
$(\rho, u, d)$ in $\Omega\times(0, T_*]$.
By the definitions of \eqref{3.5} and \eqref{a10}, it is easy to check that
\begin{align*}
0\le \rho_0\le \hat{\rho},\quad A_1(0)=0, \quad A_2(0)\le K.
\end{align*}
Therefore, there exists a $T_1\in (0, T_*]$ such that
\begin{align}\label{z4.1}
0\le \rho_0\le 2\hat{\rho},\quad A_1(T)\le 2E_0^\frac12, \quad A_2(\sigma(T))\le 4K
\end{align}
holds for $T=T_1$.

Next, we set
\begin{align}
T^*=\sup\{T|\eqref{z4.1} ~\text{holds}\}.
\end{align}
Then $T^*\ge T_1>0$. Hence, for any $0<\tau<T< T^*$ with $T$ finite, it follows from
Lemmas \ref{l30} and \ref{l3.10} that
\begin{align}\label{z4.3}
\rho\in C([0, T]; W^{1, q}),\
(\nabla u, \, \nabla^2d) \in C(\tau, T; L^q),
\end{align}
where one has taken advantage of the standard embedding
\begin{align*}
L^\infty(\tau, T; H^1)\cap H^1(\tau, T; H^{-1})\hookrightarrow C([\tau, T]; L^q), \quad {\rm for~any}~q\in (3, 6).
\end{align*}

Finally, we claim that
\begin{align}\label{z4.6}
T^*=\infty.
\end{align}
Otherwise, $T^*<\infty$. Then, by Proposition \ref{p31}, \eqref{3.7} holds for $T=T^*$.
It follows from \eqref{3.7}, \eqref{3.8}, and \eqref{z4.3} that $(\rho(x, T^*), u(x, T^*), d(x, T^*))$
satisfies
\begin{align*}
\begin{cases}
0\le \rho(x, T^*)\le \hat{\rho}, \ (\rho(x, T^*), P(\rho(x, T^*)))\in W^{1, q},\ u(x, T^*)\in H^1_{\omega},\ d(x, T^*)\in H_n^2, \\
\|\nabla u(x, T^*)\|_{L^2}^2\le M_1, \ \|\Delta d(x, T^*)\|_{L^2}^2\le M_2,
\end{cases}
\end{align*}
and $|d(x, T^*)|=1$ due to $|d_0|=1$ (see (3.7) in \cite{LY16}).
Thus, Lemma \ref{l22} implies that there exists some $T^{**}>T^*$ such that \eqref{z4.1} holds for $T=T^{**}$,
which contradicts the definition of $T^*$. As a result, \eqref{z4.6} follows. By Lemma \ref{l22} and Lemmas \ref{l30} and \ref{l3.10},
 it indicates that $(\rho, u, d)$ is in fact the unique strong solution defined in
$\Omega\times (0, T]$ for any $0<T<T^*=\infty$. The proof of Theorem \ref{thm1} is finished.
\end{proof}

\section*{Acknowledgments}
The authors would like to express their gratitude to the reviewers for careful reading and helpful suggestions which led to an improvement of the original manuscript.

%

\end{document}